\documentclass[11pt]{article}
\usepackage{amsmath}
\usepackage{amssymb}
\usepackage{amsfonts}
\usepackage{eucal}
\usepackage{graphicx,url}
\usepackage{amssymb,amsmath,amsthm}
\usepackage{color}
\usepackage{pgfplots}
\newtheorem{theorem}{Theorem}[section]
\newtheorem{proposition}[theorem]{Proposition}
\newtheorem{definition}{Definition}[section]

\newtheorem{lemma}{Lemma}[section]
\newtheorem{remark}{Remark}[section]
\newtheorem{example}{Example}
\newtheorem{condition}{Condition}[section]
\begin{document}
\title{Large deviations for method-of-quantiles estimators of one-dimensional parameters\thanks{This
work was partially supported by the MIUR Excellence Department Project awarded to the
Department of Mathematics, University of Rome Tor Vergata (CUP E83C18000100006), by University
of Rome Tor Vergata (research programme "Mission: Sustainability", project ISIDE (grant no.
E81I18000110005)), and by Istituto Nazionale di Alta Matematica (GNAMPA funds).}}
\author{Valeria Bignozzi\thanks{Dipartimento di Statistica e Metodi
Quantitativi, Universit\`{a} di Milano Bicocca, Via Bicocca degli
Arcimboldi 8, I-20126 Milano, Italia. e-mail:
\texttt{valeria.bignozzi@unimib.it}}\and Claudio
Macci\thanks{Dipartimento di Matematica, Universit\`a di Roma Tor
Vergata, Via della Ricerca Scientifica, I-00133 Roma, Italia.
e-mail: \texttt{macci@mat.uniroma2.it}}\and Lea
Petrella\thanks{Dipartimento di Metodi e Modelli per l'Economia,
il Territorio e la Finanza, Sapienza Universit\`{a} di Roma, Via
del Castro Laurenziano 9, I-00161 Roma, Italia. e-mail:
\texttt{lea.petrella@uniroma1.it}}}
\date{}
\maketitle
\begin{abstract}
\noindent We consider method-of-quantiles estimators of unknown
one-dimensional parameters, namely the analogue of
method-of-moments estimators obtained by matching empirical and
theoretical quantiles at some probability level $\lambda\in(0,1)$.
The aim is to present large deviation results for these estimators
as the sample size tends to infinity. We study in detail several
examples; for specific models we discuss the choice of the optimal
value of $\lambda$ and we compare the convergence of the
method-of-quantiles and method-of-moments estimators.\\
\ \\
\emph{AMS Subject Classification:} 60F10; 62F10; 62F12.\\
\emph{Keywords:} location parameter; methods of moments; order
statistics; scale parameter; skewness parameter.
\end{abstract}

\section{Introduction}\label{sec:introduction}
Estimation of parameters of statistical or econometric models is
one of the main concerns in the parametric inference framework.
When the probability law is specified (up to unknown parameters),
the main tool to solve this problem is the Maximum Likelihood (ML)
technique; on the other hand, whenever the assumption of a
particular distribution is too restrictive, different solutions
may be considered. For instance the Method of Moments (MM) and the
Generalized Method of Moments (GMM) provide valuable alternative
procedures; in fact the application of these methods only requires
the knowledge of some moments.

A different approach is to consider the Method of Quantiles (MQ),
that is the analogue of MM with quantiles; MQ estimators are
obtained by matching the empirical percentiles with their
theoretical counterparts at one or more probability levels.
Inference via quantiles goes back to (Aitchinson and Brown 1957)
where the authors consider an estimation problem for a
three-parameter log-normal distribution; their approach consists
in minimizing a suitable distance between the theoretical and
empirical quantiles, see for instance (Koenker 2005). Successive
papers deal with the estimation of parameters of extreme value
(see (Hassanein 1969a) and (Hassanein 1972)), logistic (see
(Hassanein 1969b)) and Weibull (see (Hassanein 1971))
distributions. A more recent reference is (Castillo and Hadi 1995)
where several other distributions are studied. We also recall
(Dominicy and Veredas 2013) where the authors consider an indirect
inference method based on the simulation of theoretical quantiles,
or a function of them, when they are not available in a closed
form. In (Sgouropoulos, Yao and Yastremiz 2015), an iterative
procedure based on ordinary least-squares estimation is proposed
to compute MQ estimators; such estimators can be easily modified
by adding a LASSO penalty term if a sparse representation is
desired, or by restricting the matching within a given range of
quantiles to match a part of the target distribution. Quantiles
and empirical quantiles represent a key tool also in quantitative
risk management, where they are studied under the name of
Value-at-Risk (see for instance (McNeil, Frey and Embrechts
2015)).

In our opinion, MQ estimators deserve a deeper investigation
because of several advantages. They allow to estimate parameters
when the moments are not available and they are invariant with
respect to increasing transformations; moreover they have less
computational problems, and behave better when distributions are
heavy-tailed or their supports vary with the parameters.

The aim of this paper is to present large deviation results for MQ
estimators (as the sample size tends to infinity) for statistical
models with one-dimensional unknown parameter $\theta\in\Theta$,
where the parameter space $\Theta$ is a subset of the real line;
thus we match empirical and theoretical quantiles at one
probability level $\lambda\in(0,1)$. The theory of large
deviations is a collection of techniques which gives an asymptotic
computation of small probabilities on an exponential scale (see
e.g. (Dembo and Zeitouni 1998) as a reference on this topic).
Several examples of statistical models are considered throughout
the paper, and some particular distributions are studied in
detail. For most of the examples considered, we are able to find
an explicit expression for the rate function which governs the
large deviation principle of the MQ estimators and, when possible,
our investigation provides the optimal $\lambda$ that guarantees a
faster convergence to the true parameter (see Definition
\ref{def:optimal-lambda}). Further we compare MQ and MM estimators
in terms of the local behavior of the rate functions around the
true value of the parameter in the spirit of Remark
\ref{rem:comparison-between-rfs}. Which one of the estimators
behaves better strictly depends on the type of parameter we have
to estimate and varies upon distributions. However, we provide
explicit examples (a part from the obvious ones where the MM
estimators are not available) where MQ estimators are preferable.

We conclude with the outline of the paper. In Section
\ref{sec:Preliminaries} we recall some preliminaries. Sections
\ref{sec:results-for-MQestimators} and
\ref{sec:results-for-MMestimators} are devoted to the results for
MQ and MM estimators, respectively. In Section \ref{sec:examples}
we present examples for different kind of parameters (e.g. scale,
location, skewness, etc.), and for each example specific
distributions are discussed in Section
\ref{sec:local-comparison-for-examples}.

\section{Preliminaries}\label{sec:Preliminaries}
In this section we present
some preliminaries on large deviations and we provide a rigorous
definition of the MQ estimators studied in this paper (see
Definition \ref{def:MQ-estimators} below).


\subsection{Large deviations}\label{sub:LD-preliminaries}
We start with the concept of large deviation principle (LDP for
short). A sequence of random variables $\{W_n:n\geq 1\}$ taking
values on a topological space $\mathcal{W}$ satisfies the LDP with
rate function $I:\mathcal{W}\to[0,\infty]$ if $I$ is a lower
semi-continuous function,
$$\liminf_{n\to\infty}\frac{1}{n}\log P(W_n\in O)\geq-\inf_{w\in O}I(w)\ \mbox{for all open sets}\ O$$
and
$$\limsup_{n\to\infty}\frac{1}{n}\log P(W_n\in C)\leq-\inf_{w\in C}I(w)\ \mbox{for all closed sets}\ C.$$
We also recall that a rate function $I$ is said to be good if all
its level sets \mbox{$\{\{w\in\mathcal{W}:I(w)\leq\eta\}:\eta\geq 0\}$}
are compact.

\begin{remark}[Local comparison between rate functions around the unique common zero]\label{rem:comparison-between-rfs}
It is known that, if $I$ uniquely vanish at some
$w_0\in\mathcal{W}$, then the sequence of random variables
converges weakly to $w_0$. Moreover, if we have two rate functions
$I_1$ and $I_2$ which uniquely vanish at the same point
$w_0\in\mathcal{W}$, and if $I_1(w)>I_2(w)>0$ for $w$ in a 
neighborhood of $w_0$ (except $w_0$) then any sequence
which satisfies the LDP with rate function $I_1$ converges to
$w_0$ faster than any sequence which satisfies the LDP with rate
function $I_2$.
\end{remark}

We also recall a recent large deviation result on order statistics
of i.i.d. random variables (see Proposition
\ref{prop:Theorem3.2-in-HMP-restricted} below) which plays a
crucial role in this paper. We start with the following condition.

\begin{condition}\label{cond:(*)inHMP-with-restriction}
Let $\{X_n:n\geq 1\}$ be a sequence of i.i.d. real valued random
variables with distribution function $F$, and assume that $F$ is
continuous and strictly increasing on $(\alpha,\omega)$, where
$-\infty\leq\alpha<\omega\leq\infty$. Moreover let $\{k_n:n\geq
1\}$ be such that $k_n\in\{1,\ldots,n\}$ for all $n\geq 1$ and
$\lim_{n\to\infty}\frac{k_n}{n}=\lambda\in(0,1)$.
\end{condition}

We introduce the following notation: for all $k\geq 1$,
$X_{1:k}\leq\cdots\leq X_{k:k}$ are the order statistics of the
sample $X_1,\ldots,X_k$; for $p,q\in(0,1)$ we set
\begin{equation}\label{eq:def-function-H}
H(p|q):=p\log\frac{p}{q}+(1-p)\log\frac{1-p}{1-q},
\end{equation}
that is the relative entropy of the Bernoulli distribution with
parameter $p$ with respect to the Bernoulli distribution with
parameter $q$.

\begin{proposition}[Theorem 3.2 in (Hashorva, Macci and Pacchiarotti 2013) for $\lambda\in(0,1)$]\label{prop:Theorem3.2-in-HMP-restricted}
Assume that Condition \ref{cond:(*)inHMP-with-restriction} holds.
Then $\{X_{k_n:n}:n\geq 1\}$ satisfies the LDP with good rate
function $I_{\lambda,F}$ defined by
$$I_{\lambda,F}(x):=\left\{\begin{array}{ll}
H(\lambda|F(x))&\ \mbox{for}\ x\in(\alpha,\omega)\\
\infty&\ \mbox{otherwise}.
\end{array}\right.$$
\end{proposition}

\begin{remark}[$I_{\lambda,F}^{\prime\prime}(F^{-1}(\lambda))$ as the inverse of an asymptotic variance]\label{rem:Dasgupta-connection}
Theorem 7.1(c) in (Dasgupta 2008) states that, under suitable
conditions, $\{\sqrt{n}(X_{k_n:n}-F^{-1}(\lambda)):n\geq 1\}$
converges weakly to the centered Normal distribution with variance
$\sigma^2:=\frac{\lambda(1-\lambda)}{(F^\prime(F^{-1}(\lambda)))^2}$.
Then, if we assume that $F$ is twice differentiable, we can check
that
$I_{\lambda,F}^{\prime\prime}(F^{-1}(\lambda))=\frac{1}{\sigma^2}$
with some computations.
\end{remark}

A more general formulation of Proposition
\ref{prop:Theorem3.2-in-HMP-restricted} could be given for
$\lambda\in[0,1]$ but, in view of the applications presented in
this paper, we prefer to consider a restricted version of the
result with $\lambda\in(0,1)$ only (so we do not consider the
cases $\lambda=0$ and $\lambda=1$). This restriction allows to
have the goodness of the rate function $I_{\lambda,F}$ (see Remark
1 in (Hashorva, Macci and Pacchiarotti 2013)) which is needed to
apply the contraction principle (see e.g. Theorem 4.2.1 in (Dembo
and Zeitouni 1998)).

\subsection{MQ estimators}\label{sub:MQ-preliminaries}
Here we present a rigorous definition of MQ estimators. In view of
this, the next Condition \ref{cond:forMQestimator} plays a crucial
role.

\begin{condition}\label{cond:forMQestimator}
Let $\{F_\theta:\theta\in\Theta\}$ be a family of distribution
functions where $\Theta\subset\mathbb{R}$ and, for all
$\theta\in\Theta$, $F_\theta$ is continuous and strictly
increasing on some $(\alpha_\theta,\omega_\theta)$, where
$-\infty\leq\alpha_\theta<\omega_\theta\leq\infty$ (as happens for
the distribution function $F$ in Condition
\ref{cond:(*)inHMP-with-restriction}). Moreover, for
$\lambda\in(0,1)$, consider the function
$$\theta\mapsto F_\theta^{-1}(\lambda)$$
(for $\theta\in\Theta$). Moreover we assume that, for all
$m\in\mathcal{M}:=\bigcup_{\theta\in\Theta}(\alpha_\theta,\omega_\theta)$,
the equation $F_\theta^{-1}(\lambda)=m$ admits a unique solution
(with respect to $\theta\in\Theta$) which will be denoted by
$\theta_\lambda(m)$.
\end{condition}

Now we are ready to present the definition.

\begin{definition}\label{def:MQ-estimators}
Assume that Condition \ref{cond:forMQestimator} holds. Then
$\left\{\theta_\lambda(X_{[\lambda n]:n}):n\geq 1\right\}$ 
is a sequence of MQ estimators (for the level $\lambda\in(0,1)$).
\end{definition}

Proposition \ref{prop:LD-for-MQ-estimators} below provides the LDP
for the sequence of estimators in Definition
\ref{def:MQ-estimators} (as the sample size $n$ goes to infinity)
when the true value of the parameter is $\theta_0\in\Theta$.
Actually we give a more general formulation in terms of
$\left\{\theta_\lambda(X_{k_n:n}):n\geq 1\right\}$, where 
$\{k_n:n\geq 1\}$ is a sequence as in Condition
\ref{cond:(*)inHMP-with-restriction}.

\section{Results for MQ estimators}\label{sec:results-for-MQestimators}
In this section we prove the LDP for the sequence of estimators in
Definition \ref{def:MQ-estimators}. Moreover we discuss some
properties of the rate function; in particular Proposition
\ref{prop:second-derivative-I} (combined with Remark
\ref{rem:comparison-between-rfs} above) leads us to define a
concept of optimal $\lambda$ presented in Definition
\ref{def:optimal-lambda} below.

We start with our main result and, in view of this, we present the
following notation:
\begin{equation}\label{eq:def-function-h}
h_{\lambda,\theta_0}(\theta):=F_{\theta_0}(F_\theta^{-1}(\lambda))\
(\mbox{for}\
F_\theta^{-1}(\lambda)\in(\alpha_{\theta_0},\omega_{\theta_0})).
\end{equation}

\begin{proposition}[LD for MQ estimators]\label{prop:LD-for-MQ-estimators}
Assume that $\{k_n:n\geq 1\}$ is as in Condition
\ref{cond:(*)inHMP-with-restriction} and that Condition
\ref{cond:forMQestimator} holds. Moreover assume that, for some
$\theta_0\in\Theta$, $\{X_n:n\geq 1\}$ are i.i.d. random variables
with distribution function $F_{\theta_0}$. Then, if the
restriction of $\theta_\lambda(\cdot)$ on
$(\alpha_{\theta_0},\omega_{\theta_0})$ is continuous,
$\left\{\theta_\lambda(X_{k_n:n}):n\geq
1\right\}$ satisfies the LDP with good rate function
$I_{\lambda,\theta_0}$ defined by
$$I_{\lambda,\theta_0}(\theta):=\left\{\begin{array}{ll}
\lambda\log\frac{\lambda}{h_{\lambda,\theta_0}(\theta)}
+(1-\lambda)\log\frac{1-\lambda}{1-h_{\lambda,\theta_0}(\theta)}&\ \mbox{for}\ \theta\in\Theta\ \mbox{such that}\ F_\theta^{-1}(\lambda)\in(\alpha_{\theta_0},\omega_{\theta_0})\\
\infty&\ \mbox{otherwise},
\end{array}\right.$$
where $h_{\lambda,\theta_0}(\theta)$ is defined by
\eqref{eq:def-function-h}.
\end{proposition}
\begin{proof}
Since the restriction of $\theta_\lambda(\cdot)$
on $(\alpha_{\theta_0},\omega_{\theta_0})$ is continuous, a
straightforward application of the contraction principle yields
the LDP of
$\left\{\theta_\lambda(X_{k_n:n}):n\geq 1\right\}$ with good rate
function $I_{\lambda,\theta_0}$ defined by
$$I_{\lambda,\theta_0}(\theta):=\inf\left\{I_{\lambda,F_{\theta_0}}(x):x\in(\alpha_{\theta_0},\omega_{\theta_0}),\theta_\lambda(x)=\theta\right\},$$ 
where $I_{\lambda,F_{\theta_0}}$ is the good rate function in
Proposition \ref{prop:Theorem3.2-in-HMP-restricted}, namely the
good rate function defined by
$I_{\lambda,F_{\theta_0}}(x):=H(\lambda|F_{\theta_0}(x))$, for
$x\in(\alpha_{\theta_0},\omega_{\theta_0})$. Moreover the set
$\left\{x\in(\alpha_{\theta_0},\omega_{\theta_0}):\theta_\lambda(x)=\theta\right\}$
has at most one element, namely
$$\left\{x\in(\alpha_{\theta_0},\omega_{\theta_0}):\theta_\lambda(x)=\theta\right\}=
\left\{\begin{array}{ll}
\{F_\theta^{-1}(\lambda)\}&\ \mbox{for}\ \theta\in\Theta\ \mbox{such that}\ F_\theta^{-1}(\lambda)\in(\alpha_{\theta_0},\omega_{\theta_0})\\
\emptyset&\ \mbox{otherwise};
\end{array}\right.$$
thus we have
$I_{\lambda,\theta_0}(\theta)=H(\lambda|F_{\theta_0}(F_\theta^{-1}(\lambda)))=H(\lambda|h_{\lambda,\theta_0}(\theta))$
for $\theta\in\Theta$ such that
$F_\theta^{-1}(\lambda)\in(\alpha_{\theta_0},\omega_{\theta_0})$,
and $I_{\lambda,\theta_0}(\theta)=\infty$ otherwise. The proof is
completed by taking into account the definition of the function $H$
in \eqref{eq:def-function-H}.
\end{proof}

\begin{remark}[Rate function invariance with respect to increasing transformations]\label{rem:invariance-wrt-increasing-transformations}
Let $\{F_\theta:\theta\in\Theta\}$ be a family of distribution
functions as in Condition \ref{cond:forMQestimator} and assume
that there exists an interval $(\alpha,\omega)$ such that
$(\alpha_\theta,\omega_\theta)=(\alpha,\omega)$ for all
$\theta\in\Theta$. Moreover let
$\psi:(\alpha,\omega)\to\mathbb{R}$ be a strictly increasing
function. Then, if we consider the MQ estimators based on the
sequence $\{\psi(X_n):n\geq 1\}$ instead of $\{X_n:n\geq 1\}$, we
can consider an adapted version of Proposition
\ref{prop:LD-for-MQ-estimators} with $(\psi(\alpha),\psi(\omega))$
in place of $(\alpha,\omega)$, $F_\theta\circ\psi^{-1}$ in place
of $F_\theta$ and, as stated in Property 1.5.16 in (Denuit et al.
2005), $\psi\circ F_\theta^{-1}$ in place of $F_\theta^{-1}$. The
LDP provided by this adapted version of Proposition
\ref{prop:LD-for-MQ-estimators} is governed by the rate function
$I_{\lambda,\theta_0;\psi}$ defined by
$$I_{\lambda,\theta_0;\psi}(\theta):=\left\{\begin{array}{ll}
H(\lambda|F_{\theta_0}\circ\psi^{-1}(\psi\circ
F_\theta^{-1}(\lambda)))
&\ \mbox{for}\ \theta\in\Theta\ \mbox{such that}\ \psi\circ F_\theta^{-1}(\lambda)\in(\psi(\alpha),\psi(\omega))\\
\infty&\ \mbox{otherwise}
\end{array}\right.$$
instead of
$$I_{\lambda,\theta_0}(\theta):=\left\{\begin{array}{ll}
H(\lambda|F_{\theta_0}(F_\theta^{-1}(\lambda)))
&\ \mbox{for}\ \theta\in\Theta\ \mbox{such that}\ F_\theta^{-1}(\lambda)\in(\alpha,\omega)\\
\infty&\ \mbox{otherwise}.
\end{array}\right.$$
One can easily realize that $I_{\lambda,\theta_0;\psi}$ and
$I_{\lambda,\theta_0}$ coincide.
\end{remark}

\begin{remark}\label{rem:motivation-of-restriction}
The proof of Proposition \ref{prop:LD-for-MQ-estimators} is based 
on Proposition \ref{prop:Theorem3.2-in-HMP-restricted}, which allows 
to consider only one-dimensional sequences of order statistics. For
this reason we focus on one-dimensional parameters only. The case of
multidimensional parameters would involve an  extended version of
Proposition \ref{prop:Theorem3.2-in-HMP-restricted} with
multidimensional sequences of order statistics, which is
non-trivial and left for future research.
\end{remark}

By taking into account the rate function in Proposition
\ref{prop:LD-for-MQ-estimators}, it would be interesting to
compare two rate functions $I_{\lambda_1,\theta_0}$ and
$I_{\lambda_2,\theta_0}$ in the spirit of Remark
\ref{rem:comparison-between-rfs} for a given pair
$\lambda_1,\lambda_2\in (0,1)$; namely it would be interesting to
have a strict inequality between $I_{\lambda_1,\theta_0}$ and
$I_{\lambda_2,\theta_0}$ in a neighborhood of $\theta_0$ (except
$\theta_0$).

Thus, if both rate functions are twice differentiable,
$I_{\lambda_1,\theta_0}$ is locally larger (resp. smaller) than
$I_{\lambda_2,\theta_0}$ around $\theta_0$ if we have
$I_{\lambda_1,\theta_0}^{\prime\prime}(\theta_0)>I_{\lambda_2,\theta_0}^{\prime\prime}(\theta_0)$
(resp.
$I_{\lambda_1,\theta_0}^{\prime\prime}(\theta_0)<I_{\lambda_2,\theta_0}^{\prime\prime}(\theta_0)$).
So it is natural to give an expression of
$I_{\lambda,\theta_0}^{\prime\prime}(\theta_0)$ under suitable
hypotheses.

\begin{proposition}[An expression for $I_{\lambda,\theta_0}^{\prime\prime}(\theta_0)$]\label{prop:second-derivative-I}
Let $I_{\lambda,\theta_0}$ be the rate function in Proposition
\ref{prop:LD-for-MQ-estimators}. Assume that $F_{\theta_0}(\cdot)$
and $F_{(\cdot)}^{-1}(\lambda)$ are twice differentiable. Then
$$I_{\lambda,\theta_0}^{\prime\prime}(\theta_0)=\frac{(h_{\lambda,\theta_0}^\prime(\theta_0))^2}{\lambda(1-\lambda)}
=\frac{\{F_{\theta_0}^\prime(F_{\theta_0}^{-1}(\lambda))\}^2}{\lambda(1-\lambda)}
\left(\left.\frac{d}{d\theta}F_\theta^{-1}(\lambda)\right|_{\theta=\theta_0}\right)^2,$$
where $h_{\lambda,\theta_0}(\theta)$ is defined by
\eqref{eq:def-function-h}.
\end{proposition}
\begin{proof}
One can easily check that
$$h_{\lambda,\theta_0}(\theta_0)=\lambda\ \mbox{and}\
h_{\lambda,\theta_0}^\prime(\theta_0)=F_{\theta_0}^\prime(F_{\theta_0}^{-1}(\lambda))\cdot
\left.\frac{d}{d\theta}F_\theta^{-1}(\lambda)\right|_{\theta=\theta_0}.$$
Moreover after some computations we get
$$I_{\lambda,\theta_0}^\prime(\theta)=h_{\lambda,\theta_0}^\prime(\theta)\left(\frac{1-\lambda}{1-h_{\lambda,\theta_0}(\theta)}
-\frac{\lambda}{h_{\lambda,\theta_0}(\theta)}\right)$$ and
$$I_{\lambda,\theta_0}^{\prime\prime}(\theta)=
h_{\lambda,\theta_0}^{\prime\prime}(\theta)\left(\frac{1-\lambda}{1-h_{\lambda,\theta_0}(\theta)}-\frac{\lambda}{h_{\lambda,\theta_0}(\theta)}\right)
+(h_{\lambda,\theta_0}^\prime(\theta))^2\left(\frac{\lambda}{h_{\lambda,\theta_0}^2(\theta)}+\frac{1-\lambda}{(1-h_{\lambda,\theta_0}(\theta))^2}\right).$$
Thus $I_{\lambda,\theta_0}^\prime(\theta_0)=0$ and
$I_{\lambda,\theta_0}^{\prime\prime}(\theta_0)=(h_{\lambda,\theta_0}^\prime(\theta_0))^2\left(\frac{1}{\lambda}+\frac{1}{1-\lambda}\right)=
\frac{(h_{\lambda,\theta_0}^\prime(\theta))^2}{\lambda(1-\lambda)}$.
The proof is completed by taking into account the expression of
$h_{\lambda,\theta_0}^\prime(\theta_0)$ above.
\end{proof}

Finally, by taking into account what we said before Proposition 
\ref{prop:second-derivative-I}, it is natural to consider the 
following

\begin{definition}\label{def:optimal-lambda}
A value $\lambda_{\mathrm{max}}\in(0,1)$ is said to be
\emph{optimal} if it maximizes
$I_{\lambda,\theta_0}^{\prime\prime}(\theta_0)$, namely if we have
$I_{\lambda_{\mathrm{max}},\theta_0}^{\prime\prime}(\theta_0)=\sup_{\lambda\in(0,1)}I_{\lambda,\theta_0}^{\prime\prime}(\theta_0)$.
\end{definition}

Note that in general $\lambda_{\mathrm{max}}$ in
Definition \ref{def:optimal-lambda} does not always exists; see
for instance Example \ref{ex:right-endpoint-parameter} presented below.

\section{Results for MM estimators}\label{sec:results-for-MMestimators}
The aim of this section is to present a version of the above
results for MM estimators; namely the LDP and an expression of
$J_{\theta_0}^{\prime\prime}(\theta_0)$, where $J_{\theta_0}$ is
the rate function which governs the LDP of MM estimators. In
particular, when we compare MM and MQ estimators in terms of speed
of convergence by referring to Remark
\ref{rem:comparison-between-rfs}, the value
$J_{\theta_0}^{\prime\prime}(\theta_0)$ will be compared with
$I_{\lambda,\theta_0}^{\prime\prime}(\theta_0)$ in Proposition
\ref{prop:second-derivative-I}.

We start with the following condition which allows us to define the
MM estimators.

\begin{condition}\label{cond:forMMestimator}
Let $\{F_\theta:\theta\in\Theta\}$ be a family of distribution
functions as in Condition \ref{cond:forMQestimator}, and
assume that it is well-defined the function
$\mu:\Theta\to\mathcal{M}$, where
$\mathcal{M}:=\bigcup_{\theta\in\Theta}(\alpha_\theta,\omega_\theta)$,
such that
$$\mu(\theta):=\int_{\alpha_\theta}^{\omega_\theta}xdF_\theta(x).$$
Moreover assume that, for all $m\in\mathcal{M}$,
the equation $\mu(\theta)=m$ admits a unique solution (with
respect to $\theta\in\Theta$) which will be denoted by
$\mu^{-1}(m)$.
\end{condition}

From now on, in connection with this condition, we introduce the
following function:
\begin{equation}\label{eq:rf-cramer-theorem}
\Lambda_\theta^*(x):=\sup_{\gamma\in\mathbb{R}}\left\{\gamma
x-\Lambda_\theta(\gamma)\right\},\ \mbox{where}\
\Lambda_\theta(\gamma):=\log\int_{\alpha_\theta}^{\omega_\theta}e^{\gamma
x}dF_\theta(x).
\end{equation}
It is well-known that, if $\{X_n:n\geq 1\}$ are i.i.d. random
variables with distribution function $F_\theta$, and if we set
$\bar{X}_n:=\frac{X_1+\cdots+X_n}{n}$ for all $n\geq 1$, then
$\{\bar{X}_n:n\geq 1\}$ satisfies the LDP with rate function
$\Lambda_\theta^*$ in \eqref{eq:rf-cramer-theorem} by Cram\'{e}r
Theorem on $\mathbb{R}$ (see e.g. Theorem 2.2.3 in (Dembo and
Zeitouni 1998)).

Then we have the following result.

\begin{proposition}[LD for MM estimators]\label{prop:LD-for-MM-estimators}
Assume that Condition \ref{cond:forMMestimator} holds. Moreover
assume that, for some $\theta_0\in\Theta$, $\{X_n:n\geq 1\}$ are
i.i.d. random variables with distribution function
$F_{\theta_0}$.\\
(i) If $\mu^{-1}(m):=c_1m+c_0$ for some $c_1,c_0\in\mathbb{R}$
such that $c_1\neq 0$, then $\{\mu^{-1}(\bar{X}_n):n\geq 1\}$
satisfies the LDP with rate function $J_{\theta_0}$ defined by
$$J_{\theta_0}(\theta):=\left\{\begin{array}{ll}
\Lambda_{\theta_0}^*(\mu(\theta))&\ \mbox{for}\ \theta\in\Theta\ \mbox{such that}\ \mu(\theta)\in(\alpha_{\theta_0},\omega_{\theta_0})\\
\infty&\ \mbox{otherwise}.
\end{array}\right.$$
(ii) If the restriction of $\mu^{-1}$ on
$(\alpha_{\theta_0},\omega_{\theta_0})$ is continuous and if
$\Lambda_{\theta_0}^*$ is a good rate function, the same LDP holds
and $J_{\theta_0}$ is a good rate function.
\end{proposition}
\begin{proof}
(i) In this case $\mu(\theta):=\frac{\theta-c_0}{c_1}$ and
$\{\mu^{-1}(\bar{X}_n):n\geq 1\}$ is again a sequence of empirical
means of i.i.d. random variables. Then the LDP still holds by
Cram\'{e}r Theorem on $\mathbb{R}$, and the rate function
$J_{\theta_0}$ is defined by
$$J_{\theta_0}(\theta):=\sup_{\gamma\in\mathbb{R}}\left\{\gamma\theta-\Lambda_{\theta_0}(c_1\gamma)-\gamma c_0\right\},$$
which yields
$$J_{\theta_0}(\theta)=\sup_{\gamma\in\mathbb{R}}\left\{c_1\gamma\frac{\theta-c_0}{c_1}-\Lambda_{\theta_0}(c_1\gamma)\right\}
=\sup_{\gamma\in\mathbb{R}}\left\{c_1\gamma\mu(\theta)-\Lambda_{\theta_0}(c_1\gamma)\right\}=\Lambda_{\theta_0}^*(\mu(\theta)),$$
as desired.\\
(ii) Since the restriction of the function $\mu^{-1}$ on
$(\alpha_{\theta_0},\omega_{\theta_0})$ is continuous and
$\Lambda_{\theta_0}^*$ is a good rate function, a straightforward
application of the contraction principle yields the LDP of
$\left\{\mu^{-1}(\bar{X}_n):n\geq 1\right\}$ with good rate
function $J_{\theta_0}$ defined by
$$J_{\theta_0}(\theta):=\inf\left\{\Lambda_{\theta_0}^*(x):x\in(\alpha_{\theta_0},\omega_{\theta_0}),\mu^{-1}(x)=\theta\right\}.$$
Moreover the set
$\left\{x\in(\alpha_{\theta_0},\omega_{\theta_0}):\mu^{-1}(x)=\theta\right\}$
has at most one element, namely
$$\left\{x\in(\alpha_{\theta_0},\omega_{\theta_0}):\mu^{-1}(x)=\theta\right\}=
\left\{\begin{array}{ll}
\{\mu(\theta)\}&\ \mbox{for}\ \theta\in\Theta\ \mbox{such that}\ \mu(\theta)\in(\alpha_{\theta_0},\omega_{\theta_0})\\
\emptyset&\ \mbox{otherwise};
\end{array}\right.$$
thus we have
$J_{\theta_0}(\theta)=\Lambda_{\theta_0}^*(\mu(\theta))$ for
$\theta\in\Theta$ such that
$\mu(\theta)\in(\alpha_{\theta_0},\omega_{\theta_0})$, and
$J_{\theta_0}(\theta)=\infty$ otherwise.
\end{proof}

Now, in the spirit of Remark \ref{rem:comparison-between-rfs}, it
would be interesting to have a local strict inequality between the
rate function $I_{\lambda,\theta_0}$ in Proposition 
\ref{prop:LD-for-MQ-estimators} for MQ estimators (for some 
$\lambda\in(0,1)$), and the rate function $J_{\theta_0}$ in 
Proposition \ref{prop:LD-for-MM-estimators} for MM estimators.

Then we can repeat the same arguments which led us to Proposition
\ref{prop:second-derivative-I}. Namely, if both rate functions
$J_{\theta_0}$ and $I_{\lambda,\theta_0}$ (for some
$\lambda\in(0,1)$) are twice differentiable, $J_{\theta_0}$ is
locally larger (resp. smaller) than $I_{\lambda,\theta_0}$ around
$\theta_0$ if
$J_{\theta_0}^{\prime\prime}(\theta_0)>I_{\lambda,\theta_0}^{\prime\prime}(\theta_0)$
(resp.
$J_{\theta_0}^{\prime\prime}(\theta_0)<I_{\lambda,\theta_0}^{\prime\prime}(\theta_0)$).
So it is natural to give an expression of
$J_{\theta_0}^{\prime\prime}(\theta_0)$ under suitable hypotheses.

\begin{proposition}[An expression for $J_{\theta_0}^{\prime\prime}(\theta_0)$]\label{prop:second-derivative-J}
Let $J_{\theta_0}$ be the rate function in Proposition
\ref{prop:LD-for-MM-estimators}. Assume that, for all
$\theta\in\Theta$, the function $\Lambda_\theta$ in
\eqref{eq:rf-cramer-theorem} is finite in a neighborhood of the
origin $\gamma=0$ and that $\mu(\cdot)$ is twice differentiable.
Then
$J_{\theta_0}^{\prime\prime}(\theta_0)=\frac{(\mu^\prime(\theta_0))^2}{\sigma^2(\theta_0)}$,
where
$\sigma^2(\theta):=\int_{\alpha_\theta}^{\omega_\theta}x^2dF_\theta(x)-\mu^2(\theta)$
is the variance function.
\end{proposition}
\begin{proof}
One can easily check that
$$J_{\theta_0}^\prime(\theta)=(\Lambda_{\theta_0}^*)^\prime(\mu(\theta))\mu^\prime(\theta)\
\mbox{and}\
J_{\theta_0}^{\prime\prime}(\theta)=(\Lambda_{\theta_0}^*)^{\prime\prime}(\mu(\theta))(\mu^\prime(\theta))^2+\mu^{\prime\prime}(\theta)
(\Lambda_{\theta_0}^*)^\prime(\mu(\theta)).$$ Then, since we have
$(\Lambda_{\theta_0}^*)^{\prime\prime}(\mu(\theta_0))=\frac{1}{\sigma^2(\theta_0)}$
and $(\Lambda_{\theta_0}^*)^\prime(\mu(\theta_0))=0$ (this
equalities are well-known, and can be easily checked), we
immediately get the desired equality
$J_{\theta_0}^{\prime\prime}(\theta_0)=\frac{(\mu^\prime(\theta_0))^2}{\sigma^2(\theta_0)}$.
\end{proof}

\begin{remark}[On the functions $\Lambda_\theta$ and $\Lambda_\theta^*$ in \eqref{eq:rf-cramer-theorem}]\label{rem:goodness-hypothesis}
The function $\Lambda_\theta$ is finite in a neighborhood of the
origin $\gamma=0$ when we deal with empirical means (of i.i.d.
random variables) with light-tailed distribution;
in this case $\Lambda_\theta^*$ is a good rate function. On the
contrary, if we deal with i.i.d. random variables with heavy
tailed distributions, the function $\Lambda_\theta$ is not finite
in a neighborhood of the origin $\gamma=0$ and $\Lambda_\theta^*$
is not good.
\end{remark}

\section{Examples}\label{sec:examples}
The aim of this section is to present several examples of
statistical models with unknown parameter $\theta\in\Theta$, where
$\Theta\subset\mathbb{R}$; in all the examples we always deal with
one-dimensional parameters assuming all the others to be known.

Let us briefly introduce the examples presented below. We
investigate distributions with scale parameter in Example
\ref{ex:scale-parameter}, with location parameter in Example
\ref{ex:location-parameter}, and with skewness parameter in
Example \ref{ex:skew-parameter}. We remark that in Example
\ref{ex:skew-parameter} we use the epsilon-Skew-Normal
distribution defined in (Mudholkar and Hutson 2000); this choice
is motivated by the
availability of an explicit expression of the inverse of the
distribution function giving us the possibility of obtaining
explicit formulas. Moreover we present Example
\ref{ex:Pareto-distributions} with Pareto distributions, which
allows to give a concrete illustration of the content of Remark
\ref{rem:invariance-wrt-increasing-transformations}. In all these
statistical models the intervals
$\{(\alpha_\theta,\omega_\theta):\theta\in\Theta\}$ do not depend
on $\theta$ and we simply write $(\alpha,\omega)$. Finally we
present Example \ref{ex:right-endpoint-parameter} where we have
$(\alpha_\theta,\omega_\theta)=(0,\theta)$ for
$\theta\in\Theta:=(0,\infty)$; namely for this example $\theta$ is
a right-endpoint parameter.

In all examples (except Example \ref{ex:Pareto-distributions}) we
give a formula for $I_{\lambda,\theta_0}^{\prime\prime}(\theta_0)$
(as a consequence of Proposition \ref{prop:second-derivative-I})
which will be used for the local comparisons between rate
functions (in the spirit of Remark
\ref{rem:comparison-between-rfs}) analyzed in Section
\ref{sec:local-comparison-for-examples}.

In what follows we say that a distribution function $F$ on
$\mathbb{R}$ has the \emph{symmetry property} if it is a
distribution function of a symmetric random variable, i.e. if
$F(x)=1-F(-x)$ for all $x\in\mathbb{R}$. In such a case we have
$F^{-1}(\lambda)=-F^{-1}(1-\lambda)$ for all $\lambda\in(0,1)$.

\begin{example}[Statistical model with a scale parameter $\theta\in\Theta:=(0,\infty)$]\label{ex:scale-parameter}
Let $F_\theta$ be defined by
$$F_\theta(x):=G\left(\frac{x}{\theta}\right)\ \mbox{for}\ x\in(\alpha,\omega),$$
where $G$ is a strictly increasing distribution function on
$(\alpha,\omega)=(0,\infty)$ or
$(\alpha,\omega)=(-\infty,\infty)$. Then
$$F_\theta^{-1}(\lambda):=\theta G^{-1}(\lambda)\ \mbox{and}
\ h_{\lambda,\theta_0}(\theta)=G\left(\frac{\theta}{\theta_0}\cdot
G^{-1}(\lambda)\right);$$ it is important to remark that, when
$(\alpha,\omega)=(-\infty,\infty)$, the value $\lambda=G(0)$
(which yields $G^{-1}(\lambda)=0$) is not allowed. Now we give a
list of some specific examples studied in this paper.

For the case $(\alpha,\omega)=(0,\infty)$ we consider the Weibull
distribution:
\begin{equation}\label{eq:df-Weibull}
G(x):=1-\exp(-x^{\rho})\ (\mbox{where}\ \rho>0)\ \mbox{and}\
G^{-1}(\lambda):=\left(-\log(1-\lambda)\right)^{1/\rho}.
\end{equation}
We also give some specific examples where
$(\alpha,\omega)=(-\infty,\infty)$ and, in each case,
$\eta\in\mathbb{R}$ is a known location parameter (and the
not-allowed value $\lambda=G(0)$ depends on $\eta$): the Normal
distribution
\begin{equation}\label{eq:df-Normal-scale-parameter}
G(x):=\Phi(x-\eta)\ \mbox{and}\
G^{-1}(\lambda):=\eta+\Phi^{-1}(\lambda),
\end{equation}
where $\Phi$ is the standard Normal distribution function;
the Cauchy distribution
\begin{equation}\label{eq:df-Cauchy-scale-parameter}
G(x):=\frac{1}{\pi}\left(\arctan(x-\eta)+\frac{\pi}{2}\right)\
\mbox{and}\
G^{-1}(\lambda):=\eta+\tan\left(\left(\lambda-\frac{1}{2}\right)\pi\right);
\end{equation}
the logistic distribution
\begin{equation}\label{eq:df-logistic-scale-parameter}
G(x):=\frac{1}{1+e^{-(x-\eta)}}\ \mbox{and}\
G^{-1}(\lambda):=\eta-\log\left(\frac{1}{\lambda}-1\right);
\end{equation}
the Gumbel distribution
\begin{equation}\label{eq:df-Gumbel-scale-parameter}
G(x):=\exp(-e^{-(x-\eta)})\ \mbox{and}\
G^{-1}(\lambda):=\eta-\log(-\log\lambda).
\end{equation}
If $G$ is twice differentiable we have
\begin{equation}\label{eq:second-derivative-scale-parameter}
I_{\lambda,\theta_0}^{\prime\prime}(\theta_0)=\frac{\{G^\prime(G^{-1}(\lambda))G^{-1}(\lambda)\}^2}{\lambda(1-\lambda)\theta_0^2}
\end{equation}
by Proposition \ref{prop:second-derivative-I}; so, if it is
possible to find an optimal $\lambda_{\mathrm{max}}$, such a value
does not depend on $\theta_0$ (on the contrary it could depend on
the known location parameter $\eta$ as we shall see in Section
\ref{sec:local-comparison-for-examples}). Moreover one can check
that $I_{\lambda,\theta_0}^{\prime\prime}(\theta_0)=0$ if we
consider the not-allowed value $\lambda=G(0)\in(0,1)$ (when
$(\alpha,\omega)=(-\infty,\infty)$) because $G^{-1}(\lambda)=0$,
and that
$I_{\lambda,\theta_0}^{\prime\prime}(\theta_0)=I_{1-\lambda,\theta_0}^{\prime\prime}(\theta_0)$
(for all $\lambda\in(0,1)$) if $G$ is symmetric as it happens, for
instance, in \eqref{eq:df-Normal-scale-parameter},
\eqref{eq:df-Cauchy-scale-parameter} and
\eqref{eq:df-logistic-scale-parameter} with $\eta=0$.
\end{example}

\begin{example}[Statistical model with a location parameter $\theta\in\Theta:=(-\infty,\infty)$]\label{ex:location-parameter}
Let $F_\theta$ be defined by
$$F_\theta(x):=G(x-\theta)\ \mbox{for}\ x\in(\alpha,\omega)=(-\infty,\infty),$$
where $G$ is a strictly increasing distribution function on
$(\alpha,\omega)=(-\infty,\infty)$. Then
$$F_\theta^{-1}(\lambda):=\theta+G^{-1}(\lambda)\ \mbox{and}
\ h_{\lambda,\theta_0}(\theta)=G\left(\theta+G^{-1}(\lambda)-\theta_0\right).$$
We give some specific examples studied in this paper and, in each
case, $s>0$ is a known scale parameter: the Normal distribution
\begin{equation}\label{eq:df-Normal-location-parameter}
G(x):=\Phi\left(\frac{x}{s}\right)\ \mbox{and}\
G^{-1}(\lambda):=s\cdot\Phi^{-1}(\lambda);
\end{equation}
the Cauchy distribution
\begin{equation}\label{eq:df-Cauchy-location-parameter}
G(x):=\frac{1}{\pi}\left(\arctan\frac{x}{s}+\frac{\pi}{2}\right)\
\mbox{and}\
G^{-1}(\lambda):=s\cdot\tan\left(\left(\lambda-\frac{1}{2}\right)\pi\right);
\end{equation}
the logistic distribution
\begin{equation}\label{eq:df-logistic-location-parameter}
G(x):=\frac{1}{1+e^{-x/s}}\ \mbox{and}\
G^{-1}(\lambda):=-s\cdot\log\left(\frac{1}{\lambda}-1\right);
\end{equation}
the Gumbel distribution
\begin{equation}\label{eq:df-Gumbel-location-parameter}
G(x):=\exp(-e^{-x/s})\ \mbox{and}\
G^{-1}(\lambda):=-s\cdot\log(-\log\lambda).
\end{equation}
If $G$ is twice differentiable we have
\begin{equation}\label{eq:second-derivative-location-parameter}
I_{\lambda,\theta_0}^{\prime\prime}(\theta_0)=\frac{\{G^\prime(G^{-1}(\lambda))\}^2}{\lambda(1-\lambda)}
\end{equation}
by Proposition \ref{prop:second-derivative-I}; so, if it is
possible to find an optimal $\lambda_{\mathrm{max}}$, such a value
does not depend on $\theta_0$ and on the known scale parameter
$s$. Moreover one can check that
$I_{\lambda,\theta_0}^{\prime\prime}(\theta_0)=I_{1-\lambda,\theta_0}^{\prime\prime}(\theta_0)$
(for all $\lambda\in(0,1)$) if $G$ has the symmetry property (as
happens for $G$ in \eqref{eq:df-Normal-location-parameter},
\eqref{eq:df-Cauchy-location-parameter} and
\eqref{eq:df-logistic-location-parameter}, and not for $G$ in
\eqref{eq:df-Gumbel-location-parameter}).
\end{example}

\begin{example}[Statistical model with a skewness parameter $\theta\in\Theta:=(-1,1)$]\label{ex:skew-parameter}
Let $F_\theta$ be defined by
$$F_\theta(x):=\left\{\begin{array}{ll}
(1+\theta)G(\frac{x}{1+\theta})&\ \mbox{for}\ x\leq 0\\
\theta+(1-\theta)G(\frac{x}{1-\theta})&\ \mbox{for}\ x>0,
\end{array}\right.\ \mbox{with}\ x\in(\alpha,\omega)=(-\infty,\infty),$$
where $G$ is a strictly increasing distribution function on
$(\alpha,\omega)=(-\infty,\infty)$ with the symmetry property.
Then
$$F_\theta^{-1}(\lambda):=\left\{\begin{array}{ll}
(1+\theta)G^{-1}(\frac{\lambda}{1+\theta})&\ \mbox{for}\ \lambda\in(0,\frac{1+\theta}{2}]\\
(1-\theta)G^{-1}(\frac{\lambda-\theta}{1-\theta})&\ \mbox{for}\
\lambda\in(\frac{1+\theta}{2},1)
\end{array}\right.$$
and
$$h_{\lambda,\theta_0}(\theta)=\left\{\begin{array}{ll}
(1+\theta_0)G\left(\frac{1+\theta}{1+\theta_0}G^{-1}(\frac{\lambda}{1+\theta})\right)&\ \mbox{for}\ \theta\geq 2\lambda-1\\
\theta_0+(1-\theta_0)G\left(\frac{1-\theta}{1-\theta_0}G^{-1}(\frac{\lambda-\theta}{1-\theta})\right)&\
\mbox{for}\ \theta<2\lambda-1.
\end{array}\right.$$
We can consider the same specific examples presented in Example
\ref{ex:location-parameter}, i.e. the functions $G$ in
\eqref{eq:df-Normal-location-parameter},
\eqref{eq:df-Cauchy-location-parameter} and
\eqref{eq:df-logistic-location-parameter} for some known scale
parameter $s>0$.

If $G$ is twice differentiable and $G^{\prime\prime}(0)=0$ we have
$$I_{\lambda,\theta_0}^{\prime\prime}(\theta_0)=\left\{\begin{array}{ll}
\frac{1}{\lambda(1-\lambda)}\left[G^\prime(G^{-1}(\frac{\lambda}{1+\theta_0}))
\left(G^{-1}(\frac{\lambda}{1+\theta_0})-\frac{\lambda}{1+\theta_0}(G^{-1})^\prime(\frac{\lambda}{1+\theta_0})\right)\right]^2
&\ \mbox{for}\ \lambda\in(0,\frac{1+\theta_0}{2}]\\
\frac{1}{\lambda(1-\lambda)}\left[G^\prime(G^{-1}(\frac{\lambda-\theta_0}{1-\theta_0}))
\left(-G^{-1}(\frac{\lambda-\theta_0}{1-\theta_0})+\frac{\lambda-1}{1-\theta_0}(G^{-1})^\prime(\frac{\lambda-\theta_0}{1-\theta_0})\right)\right]^2
&\ \mbox{for}\ \lambda\in(\frac{1+\theta_0}{2},1)
\end{array}\right.$$
by Proposition \ref{prop:second-derivative-I}, and therefore
\begin{equation}\label{eq:second-derivative-skew-parameter}
I_{\lambda,\theta_0}^{\prime\prime}(\theta_0)=\left\{\begin{array}{ll}
\frac{1}{\lambda(1-\lambda)}\left[G^\prime(G^{-1}(\frac{\lambda}{1+\theta_0}))G^{-1}(\frac{\lambda}{1+\theta_0})-\frac{\lambda}{1+\theta_0}\right]^2
&\ \mbox{for}\ \lambda\in(0,\frac{1+\theta_0}{2}]\\
\frac{1}{\lambda(1-\lambda)}\left[-G^\prime(G^{-1}(\frac{\lambda-\theta_0}{1-\theta_0}))
G^{-1}(\frac{\lambda-\theta_0}{1-\theta_0})+\frac{\lambda-1}{1-\theta_0}\right]^2
&\ \mbox{for}\ \lambda\in(\frac{1+\theta_0}{2},1);
\end{array}\right.
\end{equation}
so one can expect that, if it is possible to find an optimal
$\lambda_{\mathrm{max}}$, such a value depends on $\theta_0$ (this
is what happens in Section
\ref{sec:local-comparison-for-examples}). Moreover one can check
that $I_{\lambda,\theta_0}^{\prime\prime}(\theta_0)$ in
\eqref{eq:second-derivative-skew-parameter} does not depend on
$s$, $I_{1/2,0}^{\prime\prime}(0)=1$ and
$I_{\lambda,\theta_0}^{\prime\prime}(\theta_0)=I_{1-\lambda,-\theta_0}^{\prime\prime}(-\theta_0)$
(for all $\lambda\in(0,1)$).
\end{example}

\begin{example}[Statistical model with Pareto distributions with $\theta\in\Theta:=(0,\infty)$]\label{ex:Pareto-distributions}
Let $F_\theta$ be defined by
$$F_\theta(x):=1-x^{-1/\theta}\ \mbox{for}\ x\in(\alpha,\omega)=(1,\infty).$$
Then
\begin{equation}\label{eq:basic-positions-Pareto-distributions}
F_\theta^{-1}(\lambda):=e^{-\theta\log(1-\lambda)}=(1-\lambda)^{-\theta}\
\mbox{and}\ h_{\lambda,\theta_0}(\theta)=1-(1-\lambda)^{\theta/\theta_0}.
\end{equation}
We remark that, if we consider Example \ref{ex:scale-parameter}
with $G$ as in \eqref{eq:df-Weibull} with $\rho=1$, namely
$$\tilde{F}_\theta(x)=1-e^{-x}\ \mbox{for}\ (\tilde{\alpha},\tilde{\omega})=(0,\infty),$$
we can refer to Remark
\ref{rem:invariance-wrt-increasing-transformations} with
$$\psi(x):=e^x\ \mbox{for}\ x\in(\tilde{\alpha},\tilde{\omega}):=(0,\infty)$$
(note that
$(\psi(\tilde{\alpha}),\psi(\tilde{\omega}))=(1,\infty)=(\alpha,\omega)$).
Then, as pointed out in Remark
\ref{rem:invariance-wrt-increasing-transformations},
$I_{\lambda,\theta_0;\psi}$ and $I_{\lambda,\theta_0}$ coincide; in fact, if 
we consider
$\tilde{F}_{\theta_0}(\tilde{F}_\theta^{-1}(\lambda))=G\left(\frac{\theta}{\theta_0}\cdot
G^{-1}(\lambda)\right)$ with $G$ as in \eqref{eq:df-Weibull} with
$\rho=1$, we obtain
$\tilde{F}_{\theta_0}(\tilde{F}_\theta^{-1}(\lambda))=1-(1-\lambda)^{\theta/\theta_0}$
(which coincides with $h_{\lambda,\theta_0}(\theta)$ in
\eqref{eq:basic-positions-Pareto-distributions}).
\end{example}

\begin{example}[Statistical model with a \lq\lq right endpoint\rq\rq parameter $\theta\in\Theta:=(0,\infty)$]\label{ex:right-endpoint-parameter}
Let $F_\theta$ be defined by
$$F_\theta(x):=\frac{G(x)}{G(\theta)}\ \mbox{for}\ x\in(\alpha_\theta,\omega_\theta):=(0,\theta),$$
where $G:[0,\infty)\to[0,\infty)$ is a strictly increasing
function such that $G(0)=0$. Then
$$F_\theta^{-1}(\lambda):=G^{-1}(\lambda G(\theta))\ \mbox{and}
\ h_{\lambda,\theta_0}(\theta)=\frac{\lambda G(\theta)}{G(\theta_0)}\ (\mbox{for}\ \lambda G(\theta)\in(0,G(\theta_0))).$$ 
Moreover, after some computations, we get
$$I_{\lambda,\theta_0}(\theta):=\left\{\begin{array}{ll}
\lambda\log\frac{G(\theta_0)}{G(\theta)}
+(1-\lambda)\log\frac{(1-\lambda)G(\theta_0)}{G(\theta_0)-\lambda
G(\theta)}
&\ \mbox{for}\ 0<\theta<G^{-1}\left(\frac{G(\theta_0)}{\lambda}\right)\\
\infty&\ \mbox{otherwise}.
\end{array}\right.$$
As a specific example we can consider $G(x)=x$ (for all $x$); in
such a case $F_\theta$ is the distribution function concerning the
uniform distribution on $(0,\theta)$. Finally, if $G$ is twice
differentiable, we have
\begin{equation}\label{eq:second-derivative-right-endpoint-parameter}
I_{\lambda,\theta_0}^{\prime\prime}(\theta_0)=\frac{\lambda(G^\prime(\theta_0))^2}{(1-\lambda)G^2(\theta_0)}
\end{equation}
by Proposition \ref{prop:second-derivative-I}.
\end{example}

\section{Local comparisons between rate functions for some examples}\label{sec:local-comparison-for-examples}
In this section we analyze the examples presented in Section
\ref{sec:examples}. We consider local comparisons
between rate functions in the spirit of Remark
\ref{rem:comparison-between-rfs} and, more precisely, the
following two issues.
\begin{itemize}
\item A discussion on the choice of optimal values of $\lambda$
in order to get the best rate of convergence. More precisely we
study the behavior of
$I_{\lambda,\theta_0}^{\prime\prime}(\theta_0)$ (varying
$\lambda\in(0,1)$) in order to find an optimal
$\lambda_{\mathrm{max}}$ in the sense of Definition
\ref{def:optimal-lambda}.
\item The comparison of the convergence of the MQ estimators and
of the MM estimators. More precisely, when we deal with MM
estimators, we compare
$I_{\lambda,\theta_0}^{\prime\prime}(\theta_0)$ in Proposition
\ref{prop:second-derivative-I} (for some $\lambda\in(0,1)$) and
$J_{\theta_0}^{\prime\prime}(\theta_0)$ in Proposition
\ref{prop:second-derivative-J}; obviously, when we have an optimal
$\lambda_{\mathrm{max}}$, we take
$\lambda=\lambda_{\mathrm{max}}$. In the single case presented
below where MM estimators are not defined, we compare the
convergence of the MQ estimators and of suitable GMM estimators.
\end{itemize}

We find at least an optimal $\lambda_{\mathrm{max}}$ for all
examples except for Example \ref{ex:right-endpoint-parameter}
(where we should consider $\lambda=1$). We present several examples 
where $(\alpha,\omega)=(-\infty,\infty)$ and we see that
$I_{\lambda,\theta_0}^{\prime\prime}(\theta_0)=I_{1-\lambda,\theta_0}^{\prime\prime}(\theta_0)$
for all $\lambda\in(0,1)$. For these examples we have the optimal
value $\lambda_{\mathrm{max}}=1/2$ or, by symmetry, two distinct
optimal values $\lambda_{\mathrm{max},1}$ and
$\lambda_{\mathrm{max},2}=1-\lambda_{\mathrm{max},1}$. However, in
general, we do not expect to have at most two values for
$\lambda_{\mathrm{max}}$.

In view of what follows it is useful to consider two suitable
values $\tilde{\lambda}_1,\tilde{\lambda}_2\in(0,1)$ presented in
the next Lemma \ref{lem:particular-values}. The value
$\tilde{\lambda}_1$ appears in the computations for the Weibull
distribution in Example \ref{ex:scale-parameter} (and also in the
computations for Example \ref{ex:Pareto-distributions} as a
trivial consequence), while the $\tilde{\lambda}_2$ appears in the
computations for the Gumbel distribution in both Examples
\ref{ex:scale-parameter} and \ref{ex:location-parameter}; however,
interestingly, Lemma \ref{lem:particular-values}(iii)
states the close relationship between $\tilde{\lambda}_1$ and 
$\tilde{\lambda}_2$. The proof of Lemma \ref{lem:particular-values}
is simple, and therefore omitted.

\begin{lemma}[The values $\tilde{\lambda}_1$ and $\tilde{\lambda}_2$]\label{lem:particular-values}
The following statements hold.\\
(i) Let $f_1$ be the function defined by
$f_1(\lambda):=\frac{(1-\lambda)(\log(1-\lambda))^2}{\lambda}$.
Then $\sup_{\lambda\in(0,1)}f_1(\lambda)=f_1(\tilde{\lambda}_1)$,
where $\tilde{\lambda}_1\simeq 0.7968$ is the unique value
$(1/2,1)$ such that
$-2\tilde{\lambda}_1-\log(1-\tilde{\lambda}_1)=0$.\\
(ii) Let $f_2$ be the function defined by
$f_2(\lambda):=\frac{\lambda(\log\lambda)^2}{1-\lambda}$. Then
$\sup_{\lambda\in(0,1)}f_2(\lambda)=f_2(\tilde{\lambda}_2)$, where
$\tilde{\lambda}_2\simeq 0.2032$ is the unique value $(0,1/2)$
such that $\log\tilde{\lambda}_2+2-2\tilde{\lambda}_2=0$.\\
(iii) We have $\tilde{\lambda}_1+\tilde{\lambda}_2=1$.
\end{lemma}

\subsection{Example \ref{ex:scale-parameter}}\label{sub:example1}
In this section we consider the particular example of the Weibull
distribution with $(\alpha,\omega)=(0,\infty)$, and the particular
examples with $(\alpha,\omega)=(-\infty,\infty)$. In each part we
analyze the MQ estimators, and we conclude with the MM estimators.

\paragraph{Analysis of MQ estimators for Weibull distribution.}
Here we consider $G$ in \eqref{eq:df-Weibull}. By
\eqref{eq:second-derivative-scale-parameter} we have
$$I_{\lambda,\theta_0}^{\prime\prime}(\theta_0)=\frac{\rho^2(1-\lambda)(\log(1-\lambda))^2}{\lambda\theta_0^2}.$$
Then we have a unique optimal value
$\lambda_{\mathrm{max}}=\tilde{\lambda}_1$ (for every $\rho$ and
$\theta_0$), where $\tilde{\lambda}_1$ is the value in Lemma
\ref{lem:particular-values}(i); in fact, if we consider the
function $f_1$ in that lemma, we have
$I_{\lambda,\theta_0}^{\prime\prime}(\theta_0)=\frac{\rho^2f_1(\lambda)}{\theta_0^2}$.

\paragraph{MM versus MQ estimators for Weibull distribution.}
We start with the analysis of MM estimators. We have
$\mu(\theta):=\theta\Gamma(1+1/\rho)$, and therefore
$\mu^{-1}(m):=\frac{m}{\Gamma(1+1/\rho)}$, where $\Gamma$ is the usual Gamma function. Thus, by Proposition
\ref{prop:LD-for-MM-estimators}(i), $\{\mu^{-1}(\bar{X}_n):n\geq
1\}$ satisfies the LDP with rate function $J_{\theta_0}$ with
$c_1=\frac{1}{\Gamma(1+1/\rho)}$ and $c_0=0$. In what follows we
consider the light-tailed case $\rho\geq 1$ and the heavy-tailed
case $\rho\in(0,1)$; some more details on the exponential
distribution case $\rho=1$ are given in Remark \ref{rem:rho=1}.\\
\underline{Light-tailed case} (namely $\rho\geq 1$). The rate
functions $J_{\theta_0}$ and $\Lambda_{\theta_0}^*$ are good and
we can refer to the comparison between
$$J_{\theta_0}^{\prime\prime}(\theta_0)=\frac{\Gamma^2(1+1/\rho)}{\theta_0^2\left[\Gamma(1+2/\rho)-\Gamma^2(1+1/\rho)\right]}$$
(this value is a consequence of Proposition
\ref{prop:second-derivative-J} noting that
$\sigma^2(\theta)=\theta^2\left[\Gamma(1+2/\rho)-\Gamma^2(1+1/\rho)\right]$)
and, for the optimal value
$\lambda_{\mathrm{max}}=\tilde{\lambda}_1$ in Lemma
\ref{lem:particular-values}(i),
$$I_{\tilde{\lambda}_1,\theta_0}^{\prime\prime}(\theta_0)=\frac{(1-\tilde{\lambda}_1)(\log(1-\tilde{\lambda}_1))^2}{\tilde{\lambda}_1\theta_0^2}
=\frac{4\tilde{\lambda}_1(1-\tilde{\lambda}_1)}{\theta_0^2}.$$ 
We remark that $J_{\theta_0}^{\prime\prime}(\theta_0)$ is an
increasing function of $\rho\in(0,\infty)$; in fact, if we set
$$a(\rho):=\Gamma^\prime(1+2/\rho)\Gamma(1+1/\rho)-\Gamma(1+2/\rho)\Gamma^\prime(1+1/\rho),$$
we have $a(\rho)>0$ (noting that
$\frac{\Gamma^\prime(1+2/\rho)}{\Gamma(1+2/\rho)}>\frac{\Gamma^\prime(1+1/\rho)}{\Gamma(1+1/\rho)}$
because the digamma function
$x\mapsto\frac{\Gamma^\prime(x)}{\Gamma(x)}$ is increasing on
$(0,\infty)$) and therefore
$$\frac{d}{d\rho}\left(\frac{\Gamma^2(1+1/\rho)}{\theta_0^2\left[\Gamma(1+2/\rho)-\Gamma^2(1+1/\rho)\right]}\right)
=\frac{2\Gamma(1+1/\rho)a(\rho)}
{\rho^2\theta_0^2\left(\frac{\Gamma(1+2/\rho)}{\Gamma^2(1+1/\rho)}-1\right)^2\Gamma^4(1+1/\rho)}>0.$$
Thus, for all $\rho\geq 1$, MM estimators converge faster than
every MQ estimators because
$$J_{\theta_0}^{\prime\prime}(\theta_0)\geq\frac{1}{\theta_0^2}>
\frac{4\tilde{\lambda}_1(1-\tilde{\lambda}_1)}{\theta_0^2}=I_{\tilde{\lambda}_1,\theta_0}^{\prime\prime}(\theta_0)$$
noting that $\inf_{\rho\geq
1}J_{\theta_0}^{\prime\prime}(\theta_0)=\left.J_{\theta_0}^{\prime\prime}(\theta_0)\right|_{\rho=1}=\frac{1}{\theta_0^2}$
and $4\tilde{\lambda}_1(1-\tilde{\lambda}_1)\simeq 0.6476$.\\
\underline{Heavy-tailed case} (namely $\rho\in(0,1)$). The rate
functions $J_{\theta_0}$ and $\Lambda_{\theta_0}^*$ are not good
(we recall Remark \ref{rem:goodness-hypothesis} presented above).
We can say that $J_{\theta_0}(\theta)=0$ for $\theta\geq\theta_0$;
thus $I_{\lambda,\theta_0}(\theta)>J_{\theta_0}(\theta)$ for
$\theta>\theta_0$ (for all $\lambda\in(0,1)$). Then we have to
compare $I_{\tilde{\lambda}_1,\theta_0}(\theta)$ and
$J_{\theta_0}(\theta)$ in a left neighborhood of $\theta_0$,
namely when $\theta\in(\theta_0-\delta,\theta_0)$ for $\delta>0$
small enough. Therefore it suffices to compare
$I_{\tilde{\lambda}_1,\theta_0}^{\prime\prime}(\theta_0)$ and the
left second derivative
$\left.\frac{d^2}{d\theta^2}J_{\theta_0}(\theta-)\right|_{\theta=\theta_0}$
which coincides with $J_{\theta_0}^{\prime\prime}(\theta_0)$
presented above for the light-tailed case. We already explained
that
$\left.\frac{d^2}{d\theta^2}J_{\theta_0}(\theta-)\right|_{\theta=\theta_0}$
is an increasing function of $\rho\in(0,\infty)$; moreover
$$\left.\frac{d^2}{d\theta^2}J_{\theta_0}(\theta-)\right|_{\theta=\theta_0}
=\frac{1}{\theta_0^2\left[\frac{\Gamma(1+2/\rho)}{\Gamma^2(1+1/\rho)}-1\right]}\to
0\ \mbox{as}\ \rho\to 0$$ by taking into account the asymptotic
behavior of Gamma function. In conclusion there exists
$\rho_0\simeq 0.81068$ (computed numerically) such that:
\begin{enumerate}
\item
$I_{\tilde{\lambda}_1,\theta_0}^{\prime\prime}(\theta_0)>\left.\frac{d^2}{d\theta^2}J_{\theta_0}(\theta-)\right|_{\theta=\theta_0}$
for $\rho\in (0,\rho_0)$;
\item
$I_{\tilde{\lambda}_1,\theta_0}^{\prime\prime}(\theta_0)<\left.\frac{d^2}{d\theta^2}J_{\theta_0}(\theta-)\right|_{\theta=\theta_0}$
for $\rho\in (\rho_0,1)$.
\end{enumerate}
Thus, by taking into account Remark
\ref{rem:comparison-between-rfs}, MQ estimators (with
$\lambda=\tilde{\lambda}_1$) converge faster than MM estimators in
the first case while, in the second case, the convergence of MQ
and MM estimators cannot be compared because we cannot find
$\delta>0$ such that
$I_{\tilde{\lambda}_1,\theta_0}(\theta)>J_{\theta_0}(\theta)$ or
$I_{\tilde{\lambda}_1,\theta_0}(\theta_)<J_{\theta_0}(\theta)$ for
$0<|\theta-\theta_0|<\delta$.

\begin{remark}[The case $\rho=1$, namely the exponential distribution]\label{rem:rho=1}
If $\rho=1$ the rate functions $J_{\theta_0}$ and
$\Lambda_{\theta_0}^*$ coincide (in fact $\Gamma(1+1/\rho)=1$) and
an explicit expression of $\Lambda_{\theta_0}^*$ is available,
namely
$$J_{\theta_0}(\theta)=\Lambda_{\theta_0}^*(\theta)=\left\{\begin{array}{ll}
\frac{\theta}{\theta_0}-1-\log\left(\frac{\theta}{\theta_0}\right)&\ \mbox{for}\ \theta\in(0,\infty)\\
\infty&\ \mbox{otherwise};
\end{array}\right.$$
then we can directly compute
$J_{\theta_0}^{\prime\prime}(\theta_0)=\frac{1}{\theta_0^2}$,
which meets the above expression. In this case the MM estimators
coincide with the ML estimators, and we already expected that they
converge faster than the MQ estimators.
\end{remark}

\paragraph{Analysis of MQ estimators for particular examples with $(\alpha,\omega)=(-\infty,\infty)$.}
Here we present the results concerning the specific examples
listed above. For all cases except the one with the Gumbel
distribution we choose $\eta$ in order to have
$G(0)\in\{0.25,0.5,0.75\}$, and we have some common features:
$G(0)=1/2$ for $\eta=0$ (actually the symmetry property holds); we
obtain symmetric values with respect $\lambda=1/2$ such that the
more the tails of the distributions are light, the more the
numerical values of $\lambda_{\mathrm{max}}$ are distant from
$\lambda=1/2$. The case with the Gumbel distribution behaves
differently because the symmetry property fails for each fixed
value of $\eta$. In all cases we can only give numerical values.\\
\underline{Normal distribution} (namely $G$ in
\eqref{eq:df-Normal-scale-parameter}). We have $G(0)=\Phi(-\eta)$
and, by \eqref{eq:second-derivative-scale-parameter},
$$I_{\lambda,\theta_0}^{\prime\prime}(\theta_0)=\frac{(\varphi(\Phi^{-1}(\lambda))\{\eta+\Phi^{-1}(\lambda)\})^2}{\lambda(1-\lambda)\theta_0^2},$$
where $\varphi$ is the standard Normal probability density function.
Moreover
$$\left.\begin{array}{lc}
&\ \mbox{numerical values for}\ \lambda_{\mathrm{max}}\\
\eta=0&\ 0.06\ (\mbox{and}\ 0.94\ \mbox{by symmetry})\\
\eta=-\Phi^{-1}(1/4)&\ 0.90\\
\eta=-\Phi^{-1}(3/4)&\ 1-0.90=0.10\ (\mbox{by symmetry})
\end{array}\right.$$
\underline{Cauchy distribution} (namely $G$ in
\eqref{eq:df-Cauchy-scale-parameter}). We have
$G(0)=\frac{1}{\pi}\left(\arctan(-\eta)+\frac{\pi}{2}\right)$ and,
by \eqref{eq:second-derivative-scale-parameter},
$$I_{\lambda,\theta_0}^{\prime\prime}(\theta_0)=\frac{\left\{\eta+\tan\left(\left(\lambda-\frac{1}{2}\right)\pi\right)\right\}^2}
{\pi^2\left\{1+\tan^2\left(\left(\lambda-\frac{1}{2}\right)\pi\right)\right\}^2\lambda(1-\lambda)\theta_0^2}.$$
Moreover
$$\left.\begin{array}{lc}
&\ \mbox{numerical values for}\ \lambda_{\mathrm{max}}\\
\eta=0&\ 0.21\ (\mbox{and}\ 0.79\ \mbox{by symmetry})\\
\eta=1&\ 0.65\\
\eta=-1&\ 1-0.65=0.35\ (\mbox{by symmetry})
\end{array}\right.$$
\underline{Logistic distribution} (namely $G$ in
\eqref{eq:df-logistic-scale-parameter}). We have
$G(0)=\frac{1}{1+e^\eta}$ and, by
\eqref{eq:second-derivative-scale-parameter},
$$I_{\lambda,\theta_0}^{\prime\prime}(\theta_0)=\frac{\lambda(1-\lambda)\left(\eta-\log(\frac{1}{\lambda}-1)\right)^2}{\theta_0^2}.$$
Moreover
$$\left.\begin{array}{lc}
&\ \mbox{numerical values for}\ \lambda_{\mathrm{max}}\\
\eta=0&\ 0.08\ (\mbox{and}\ 0.92\ \mbox{by symmetry})\\
\eta=\log 3&\ 0.85\\
\eta=-\log 3&\ 1-0.85=0.15\ (\mbox{by symmetry})
\end{array}\right.$$
\underline{Gumbel distribution} (namely $G$ in
\eqref{eq:df-Gumbel-scale-parameter}). We have
$G(0)=\exp(-e^\eta)$ and, by
\eqref{eq:second-derivative-scale-parameter},
$$I_{\lambda,\theta_0}^{\prime\prime}(\theta_0)=\frac{\lambda(\log\lambda)^2(\eta-\log(-\log\lambda))^2}{(1-\lambda)\theta_0^2}.$$
Some numerical inspections reveal that in general, for each fixed
value of $\eta$, we can find an optimal value
$\lambda_{\mathrm{max}}=\lambda_{\mathrm{max}}(\eta)$. Then, if we
consider the value $\tilde{\lambda}_2$ and the function $f_2$ in
Lemma \ref{lem:particular-values}(ii), we can say that
$$\lambda_{\mathrm{max}}=\lambda_{\mathrm{max}}(\eta)\to\tilde{\lambda}_2\ \mbox{as}\ |\eta|\to\infty$$
because $I_{\lambda,\theta_0}^{\prime\prime}(\theta_0)$ behaves
like $\frac{f_2(\lambda)\eta^2}{\theta_0^2}$ when $|\eta|$ is
large.

\paragraph{MM versus MQ estimators for particular examples with $(\alpha,\omega)=(-\infty,\infty)$.}
The MM estimators are well-defined only for the case with Gumbel
distribution and, in the spirit of Remark
\ref{rem:comparison-between-rfs}, we can compare
$J_{\theta_0}^{\prime\prime}(\theta_0)$ and
$I_{\lambda_{\mathrm{max}},\theta_0}^{\prime\prime}(\theta_0)$.
However, for Normal and logistic distributions, it is possible to
consider suitable GMM estimators $\{\tilde{\Theta}_n:n\geq 1\}$ by
matching empirical and theoretical variances; so we present the
rate function $\tilde{J}_{\theta_0}$ which governs the LDP of
$\{\tilde{\Theta}_n:n\geq 1\}$ and, at least for the case of
Normal distribution, we can give an expression of
$\tilde{J}_{\theta_0}^{\prime\prime}(\theta_0)$ and we can compare
the convergence of MQ and GMM estimators.\\
\underline{Gumbel distribution}. We have
$\mu(\theta):=\eta+\theta\gamma_*$, where $\gamma_*$ is the
Euler's constant, and therefore
$\mu^{-1}(m):=\frac{m-\eta}{\gamma_*}$. Thus, by Proposition
\ref{prop:LD-for-MM-estimators}(i), $\{\mu^{-1}(\bar{X}_n):n\geq
1\}$ satisfies the LDP with rate function $J_{\theta_0}$ with
$c_1=\frac{1}{\gamma_*}$ and $c_0=\frac{-\eta}{\gamma_*}$. The
rate functions $J_{\theta_0}$ and $\Lambda_{\theta_0}^*$ are good
(we take into account Remark \ref{rem:goodness-hypothesis}) and we
can refer to the comparison between
$I_{\lambda_{\mathrm{max}},\theta_0}^{\prime\prime}(\theta_0)$ and
$J_{\theta_0}^{\prime\prime}(\theta_0)$. We remark that, by
Proposition \ref{prop:second-derivative-J}, we have
$J_{\theta_0}^{\prime\prime}(\theta_0)=\frac{6\gamma_*^2}{\theta_0^2\pi^2}$
for each fixed value of $\eta$ (in fact we have
$\sigma^2(\theta)=\frac{\theta^2\pi^2}{6}$). Then, for all
$\eta\in\mathbb{R}$, we have
$I_{\lambda_{\mathrm{max}},\theta_0}^{\prime\prime}(\theta_0)>J_{\theta_0}^{\prime\prime}(\theta_0)$
noting that, for $\tilde{\lambda}_1$ and $\tilde{\lambda}_2$ as in
Lemma \ref{lem:particular-values}, we have
$\max\{I_{\tilde{\lambda}_1,\theta_0}^{\prime\prime}(\theta_0),I_{\tilde{\lambda}_2,\theta_0}^{\prime\prime}(\theta_0)\}>J_{\theta_0}^{\prime\prime}(\theta_0)$
(see Figure \ref{fig1}).
\begin{figure}[ht]
\begin{center}
\includegraphics[angle=0,width=0.5\textwidth]{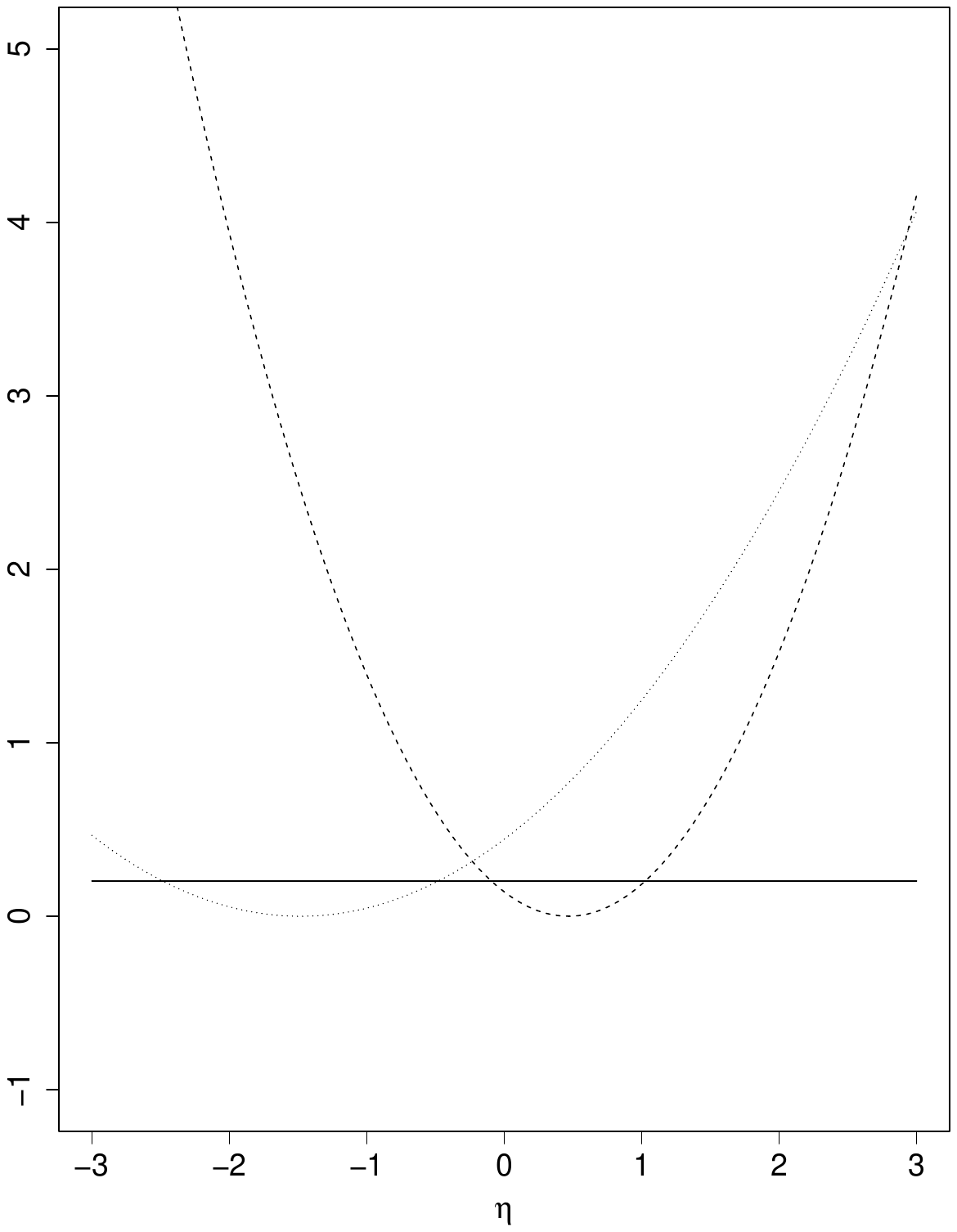}
\caption{The second derivatives
$I_{\tilde{\lambda}_2,\theta_0}^{\prime\prime}(\theta_0)$
(dashed line) and $I_{\tilde{\lambda}_1,\theta_0}^{\prime\prime}(\theta_0)$ 
(dotted line) as functions of $\eta$. The solid line represents the value
of $J_{\theta_0}^{\prime\prime}(\theta_0)$ which does not depend
on $\eta$.}\label{fig1}
\end{center}
\end{figure}\\
%
%
%
\underline{Normal (and logistic) distribution}. The MM estimators
are not well-defined because $\mu(\theta):=\eta$. So it is natural
to match empirical and theoretical variances, i.e.
$$\sigma^2(\theta)=\frac{1}{n}\sum_{i=1}^n(x_i-\eta)^2,\ \mbox{where}\ \sigma^2(\theta)=c\theta^2\ \mbox{and}\ c=\left\{\begin{array}{ll}
1&\ \mbox{for the Normal distribution}\\
\frac{\pi^2}{3}&\ \mbox{for the logistic distribution},
\end{array}\right.$$
and we obtain the GMM estimators $\{\tilde{\Theta}_n:n\geq 1\}$
defined by
$$\tilde{\Theta}_n:=\left(\frac{1}{cn}\sum_{i=1}^n(x_i-\eta)^2\right)^{1/2}.$$
Then, by adapting the proof of Proposition
\ref{prop:LD-for-MM-estimators}, we can consider the function
$$\tilde{\Lambda}_{\theta_0}^*(y):=\sup_{\gamma\in\mathbb{R}}\left\{\gamma
y-\tilde{\Lambda}_{\theta_0}(\gamma)\right\},\ \mbox{where}\
\tilde{\Lambda}_{\theta_0}(\gamma):=\log\int_{\alpha_\theta}^{\omega_\theta}e^{\gamma
(x-\eta)^2}dF_{\theta_0}(x),$$ and we can say
$\{\tilde{\Theta}_n:n\geq 1\}$ satisfies the LDP with good rate
function $\tilde{J}_{\theta_0}$ defined by
$$\tilde{J}_{\theta_0}(\theta):=\inf\{\tilde{\Lambda}_{\theta_0}^*(y):(y/c)^{1/2}=\theta\}.$$
From now on we restrict the attention to the case with Normal
distribution because we can give explicit formulas. We have
$$\tilde{\Lambda}_{\theta_0}(\gamma)=\left\{\begin{array}{ll}
\frac{1}{2}\log\left(\frac{\theta_0^2/2}{\theta_0^2/2-\gamma}\right)&\ \mbox{if}\ \gamma<\frac{\theta_0^2}{2}\\
\infty&\ \mbox{if}\ \gamma\geq\frac{\theta_0^2}{2},
\end{array}\right.\ \tilde{\Lambda}_{\theta_0}^*(y)=\left\{\begin{array}{ll}
\frac{1}{2}\left[\frac{y}{\theta_0^2}-1-\log\left(\frac{y}{\theta_0^2}\right)\right]&\ \mbox{if}\ y>0\\
\infty&\ \mbox{if}\ y\leq 0,
\end{array}\right.$$
and
$$\tilde{J}_{\theta_0}(\theta)=\left\{\begin{array}{ll}
\frac{1}{2}\left[\frac{\theta^2}{\theta_0^2}-1-\log\left(\frac{\theta^2}{\theta_0^2}\right)\right]&\ \mbox{if}\ \theta>0\\
\infty&\ \mbox{if}\ \theta\leq 0;
\end{array}\right.$$
thus, after some computations, we get
$\tilde{J}_{\theta_0}^{\prime\prime}(\theta_0)=\frac{2}{\theta_0^2}$
for all value of $\eta\in\mathbb{R}$. We conclude with the
comparison between MQ and GMM estimators. Some numerical
inspections reveal that in general, for each fixed value of
$\eta$, we can find an optimal value
$\lambda_{\mathrm{max}}=\lambda_{\mathrm{max}}(\eta)$ (their
numerical values for $\eta=0$, $\eta=\Phi^{-1}(1/4)$ and
$\eta=\Phi^{-1}(3/4)$ were presented above); moreover
$I_{\lambda_{\mathrm{max}}(\eta),\theta_0}^{\prime\prime}(\theta_0)>\tilde{J}_{\theta_0}^{\prime\prime}(\theta_0)$
for $|\eta|$ large enough because, for each fixed
$\lambda\in(0,1)$,
$$I_{\lambda,\theta_0}^{\prime\prime}(\theta_0)\to\infty\ \mbox{as}\ |\eta|\to\infty.$$
On the other hand we cannot say that
$I_{\lambda_{\mathrm{max}}(\eta),\theta_0}^{\prime\prime}(\theta_0)>\tilde{J}_{\theta_0}^{\prime\prime}(\theta_0)$
for all $\eta\in\mathbb{R}$; in fact, for $\eta=0$, we have
$$I_{\lambda_{\mathrm{max}}(0),\theta_0}^{\prime\prime}(\theta_0)\simeq\frac{0.6085}{\theta_0^2}<\frac{2}{\theta_0^2}=\tilde{J}_{\theta_0}^{\prime\prime}(\theta_0)$$
(where $\lambda_{\mathrm{max}}(0)\simeq 0.06$ or
$\lambda_{\mathrm{max}}(0)\simeq 0.94$). For completeness,
following the same lines of the particular case with Gumbel
distribution, we remark that
$$\lambda_{\mathrm{max}}=\lambda_{\mathrm{max}}(\eta)\to 1/2\ \mbox{as}\ |\eta|\to\infty$$
because, if we consider the function
$f(\lambda):=\frac{(\varphi(\Phi^{-1}(\lambda)))^2}{\lambda(1-\lambda)}$,
$I_{\lambda,\theta_0}^{\prime\prime}(\theta_0)$ behaves like
$\frac{f(\lambda)\eta^2}{\theta_0^2}$ when $|\eta|$ is large, and
$\sup_{\lambda\in(0,1)}f(\lambda)=f(1/2)$.

\subsection{Example \ref{ex:location-parameter}}\label{sub:example2}
We start with the analysis of MQ estimators. We conclude with the
MM estimators, and their comparison with the MQ estimators.

\paragraph{Analysis of MQ estimators.}
Here we present the results concerning the specific examples
listed above. In all cases, except the one with Gumbel
distribution, we can conclude that $\lambda=1/2$ is optimal;
however we can find counterexamples (see Appendix \ref{appendix}).
A further common feature (for all cases except the one with Gumbel
distribution) is that
$\left.\frac{d}{d\lambda}I_{\lambda,\theta_0}^{\prime\prime}(\theta_0)\right|_{\lambda=1/2}=0$
(and obviously this does not guarantee that $\lambda=1/2$ is an
optimal; this will be explained in Appendix \ref{appendix}); in
fact, after some computations, we can verify that
$\left.\frac{d}{d\lambda}I_{\lambda,\theta_0}^{\prime\prime}(\theta_0)\right|_{\lambda=1/2}=0$
noting that $G^{-1}(1/2)=0$ (because the distribution function $G$
has the symmetry property) and $G^{\prime\prime}(0)=0$ (because
the probability density function $G^\prime(x)$ has a maximum at
$x=0$).\\
\underline{Normal distribution} (namely $G$ in
\eqref{eq:df-Normal-location-parameter}). By
\eqref{eq:second-derivative-location-parameter} we have
$$I_{\lambda,\theta_0}^{\prime\prime}(\theta_0)=\frac{\varphi^2(\Phi^{-1}(\lambda))}{s^2\lambda(1-\lambda)}.$$
One can check numerically that we have a unique optimal
$\lambda_{\mathrm{max}}$ (for every $s$), namely
$\lambda_{\mathrm{max}}=0.5$.\\
\underline{Cauchy distribution} (namely $G$ in
\eqref{eq:df-Cauchy-location-parameter}). By
\eqref{eq:second-derivative-location-parameter} we have
$$I_{\lambda,\theta_0}^{\prime\prime}(\theta_0)=\frac{1}{s^2\pi^2\left\{1+\tan^2\left(\left(\lambda-\frac{1}{2}\right)\pi\right)\right\}^2\lambda(1-\lambda)}.$$
One can check numerically that we have a unique optimal
$\lambda_{\mathrm{max}}$ (for every $s$), namely
$\lambda_{\mathrm{max}}=0.5$.\\
\underline{Logistic distribution} (namely $G$ in
\eqref{eq:df-logistic-location-parameter}). By
\eqref{eq:second-derivative-location-parameter} we have
$$I_{\lambda,\theta_0}^{\prime\prime}(\theta_0)=\frac{\lambda(1-\lambda)}{s^2}$$
One can immediately check (we have a polynomial with degree 2)
that we have a unique optimal $\lambda_{\mathrm{max}}$ (for every
$s$), namely $\lambda_{\mathrm{max}}=0.5$.\\
\underline{Gumbel distribution} (namely $G$ in
\eqref{eq:df-Gumbel-location-parameter}). By
\eqref{eq:second-derivative-location-parameter} we have
$$I_{\lambda,\theta_0}^{\prime\prime}(\theta_0)=\frac{\lambda(\log\lambda)^2}{s^2(1-\lambda)}$$
Then we have a unique optimal value
$\lambda_{\mathrm{max}}=\tilde{\lambda}_2$ (for every $s$), where
$\tilde{\lambda}_2$ is the value in Lemma
\ref{lem:particular-values}(ii); in fact, if we consider the
function $f_2$ in that lemma, we have
$I_{\lambda,\theta_0}^{\prime\prime}(\theta_0)=\frac{f_2(\lambda)}{s^2}$.

\paragraph{MM versus MQ estimators.}
The MM estimators are well-defined in all cases except the one
with Cauchy distribution. Moreover, by taking into account Remark
\ref{rem:goodness-hypothesis}, we can always refer to the
comparison between $J_{\theta_0}^{\prime\prime}(\theta_0)$ and
$I_{\lambda_{\mathrm{max}},\theta_0}^{\prime\prime}(\theta_0)$.\\
\underline{Normal distribution}. In this case
$\mu(\theta):=\theta$. Thus, by Proposition
\ref{prop:LD-for-MM-estimators}(i), $\{\mu^{-1}(\bar{X}_n):n\geq
1\}=\{\bar{X}_n:n\geq 1\}$ satisfies the LDP with rate function
$J_{\theta_0}$ defined by
$$J_{\theta_0}(\theta)=\frac{(\theta-\theta_0)^2}{2s^2}$$
(in fact $J_{\theta_0}$ coincides with $\Lambda_{\theta_0}^*$ because 
$c_1=1$ and $c_0=0$). Then, since $\lambda_{\mathrm{max}}=1/2$ we 
have to compare $J_{\theta_0}^{\prime\prime}(\theta_0)=\frac{1}{s^2}$ 
(which meets the expression provided by Proposition
\ref{prop:second-derivative-J} noting that $\sigma^2(\theta)=s^2$)
and $I_{1/2,\theta_0}^{\prime\prime}(\theta_0)=\frac{2}{\pi s^2}$
and, obviously, we have
$J_{\theta_0}^{\prime\prime}(\theta_0)>I_{1/2,\theta_0}^{\prime\prime}(\theta_0)$
(for every $s$). Thus MM estimators converge faster than every MQ
estimators; in some sense we already expected this noting that the
MM estimators coincide with the ML estimators.\\
\underline{Logistic distribution}. In this case
$\mu(\theta):=\theta$. Thus, by Proposition
\ref{prop:LD-for-MM-estimators}(i), $\{\mu^{-1}(\bar{X}_n):n\geq
1\}=\{\bar{X}_n:n\geq 1\}$ satisfies the LDP with rate function
$J_{\theta_0}$, which coincides with $\Lambda_{\theta_0}^*$ (we
have again $c_1=1$ and $c_0=0$); in this case we cannot provide an
explicit expression of $\Lambda_{\theta_0}^*$. Then, since
$\lambda_{\mathrm{max}}=1/2$ we have to compare
$J_{\theta_0}^{\prime\prime}(\theta_0)=\frac{3}{\pi^2s^2}$ (this
is a consequence of Proposition \ref{prop:second-derivative-J}
noting that $\sigma^2(\theta)=\frac{\pi^2s^2}{3}$) and
$I_{1/2,\theta_0}^{\prime\prime}(\theta_0)=\frac{1}{4s^2}$ and,
obviously, we have
$J_{\theta_0}^{\prime\prime}(\theta_0)>I_{1/2,\theta_0}^{\prime\prime}(\theta_0)$
(for every $s$). Thus MM estimators converge faster than every MQ
estimators but, differently from what happens for the case with
Normal distribution, they do not coincide with ML estimators.\\
\underline{Gumbel distribution}. In this case
$\mu(\theta):=\theta+s\gamma_*$, where $\gamma_*$ is the Euler's
constant. Thus, by Proposition \ref{prop:LD-for-MM-estimators}(i),
$\{\mu^{-1}(\bar{X}_n):n\geq 1\}=\{\bar{X}_n:n\geq 1\}$ satisfies
the LDP with rate function $J_{\theta_0}$ (with $c_1=1$ and
$c_0=-s\gamma_*$); in this case we cannot provide an explicit
expression of $\Lambda_{\theta_0}^*$. The rate functions
$J_{\theta_0}$ and $\Lambda_{\theta_0}^*$ are good and we can
refer to the comparison between
$$J_{\theta_0}^{\prime\prime}(\theta_0)=\frac{6}{\pi^2s^2}$$
(this value is a consequence of Proposition
\ref{prop:second-derivative-J} noting that
$\sigma^2(\theta)=\frac{\pi^2s^2}{6}$) and, for the optimal value
$\lambda_{\mathrm{max}}=\tilde{\lambda}_2$ defined in Lemma
\ref{lem:particular-values}(ii),
$$I_{\tilde{\lambda}_2,\theta_0}^{\prime\prime}(\theta_0)=\frac{\tilde{\lambda}_2(\log\tilde{\lambda}_2)^2}{s^2(1-\tilde{\lambda}_2)}
=\frac{4\tilde{\lambda}_2(1-\tilde{\lambda}_2)}{s^2}.$$
We can check numerically that
$I_{\tilde{\lambda}_2,\theta_0}^{\prime\prime}(\theta_0)>J_{\theta_0}^{\prime\prime}(\theta_0)$
(for every $s$); in fact we have
$4\tilde{\lambda}_2(1-\tilde{\lambda}_2)\simeq 0.6476$ (we get a
numerical value obtained for the statistical model with Weibull
distributions because
$4\tilde{\lambda}_2(1-\tilde{\lambda}_2)=4\tilde{\lambda}_1(1-\tilde{\lambda}_1)$
by Lemma \ref{lem:particular-values}(iii)) and
$\frac{6}{\pi^2}\simeq 0.6079$. Thus MQ estimators with the
optimal value $\lambda_{\mathrm{max}}$ converge faster than MM
estimators.

\subsection{Example \ref{ex:skew-parameter}}\label{sub:example3}
Here we analyze the MQ estimators for the specific examples listed
above. In all cases we can only give numerical values; such values
depend on the unknown parameter $\theta_0$, and therefore we do
not discuss the comparison with the MM estimators (as we do for
the other examples). We have the same feature highlighted for
Example \ref{ex:scale-parameter} with
$(\alpha,\omega)=(-\infty,\infty)$, namely the more the tail of
the distributions are light, the more the numerical values of
$\lambda_{\mathrm{max}}$ are distant from $\lambda=1/2$.\\
\underline{Normal distribution} (namely $G$ in
\eqref{eq:df-Normal-location-parameter}). By
\eqref{eq:second-derivative-skew-parameter} we have
$$I_{\lambda,\theta_0}^{\prime\prime}(\theta_0)=\left\{\begin{array}{ll}
\frac{1}{\lambda(1-\lambda)}\left[\varphi(\Phi^{-1}(\frac{\lambda}{1+\theta_0}))\Phi^{-1}(\frac{\lambda}{1+\theta_0})-\frac{\lambda}{1+\theta_0}\right]^2
&\ \mbox{for}\ \lambda\in(0,\frac{1+\theta_0}{2}]\\
\frac{1}{\lambda(1-\lambda)}\left[-\varphi(\Phi^{-1}(\frac{\lambda-\theta_0}{1-\theta_0}))\Phi^{-1}(\frac{\lambda-\theta_0}{1-\theta_0})
+\frac{\lambda-1}{1-\theta_0}\right]^2 &\ \mbox{for}\
\lambda\in(\frac{1+\theta_0}{2},1).
\end{array}\right.$$
Moreover
$$\left.\begin{array}{lc}
&\ \mbox{numerical values for}\ \lambda_{\mathrm{max}}\\
\theta_0=0&\ 0.15\ (\mbox{and}\ 0.85\ \mbox{by symmetry})\\
\theta_0=1/2&\ 0.94\\
\theta_0=-1/2&\ 1-0.94=0.06\ (\mbox{by symmetry})
\end{array}\right.$$
\underline{Cauchy distribution} (namely $G$ in
\eqref{eq:df-Cauchy-location-parameter}). By
\eqref{eq:second-derivative-skew-parameter} we have
$$I_{\lambda,\theta_0}^{\prime\prime}(\theta_0)=\left\{\begin{array}{ll}
\frac{1}{\lambda(1-\lambda)}\left[\frac{\tan((\frac{\lambda}{1+\theta_0}-\frac{1}{2})\pi)}
{\pi\left(1+\tan^2((\frac{\lambda}{1+\theta_0}-\frac{1}{2})\pi)\right)}-\frac{\lambda}{1+\theta_0}\right]^2
&\ \mbox{for}\ \lambda\in(0,\frac{1+\theta_0}{2}]\\
\frac{1}{\lambda(1-\lambda)}\left[-\frac{\tan((\frac{\lambda-\theta_0}{1-\theta_0}-\frac{1}{2})\pi)}
{\pi\left(1+\tan^2((\frac{\lambda-\theta_0}{1-\theta_0}-\frac{1}{2})\pi)\right)}+\frac{\lambda-1}{1-\theta_0}\right]^2
&\ \mbox{for}\ \lambda\in(\frac{1+\theta_0}{2},1).
\end{array}\right.$$
Moreover
$$\left.\begin{array}{lc}
&\ \mbox{numerical values for}\ \lambda_{\mathrm{max}}\\
\theta_0=0&\ 0.39\ (\mbox{and}\ 0.61\ \mbox{by symmetry})\\
\theta_0=1/2&\ 0.84\\
\theta_0=-1/2&\ 1-0.84=0.16\ (\mbox{by symmetry})
\end{array}\right.$$
\underline{Logistic distribution} (namely $G$ in
\eqref{eq:df-logistic-location-parameter}). By
\eqref{eq:second-derivative-skew-parameter} we have
$$I_{\lambda,\theta_0}^{\prime\prime}(\theta_0)=\left\{\begin{array}{ll}
\frac{1}{\lambda(1-\lambda)}\left[-\frac{\lambda}{1+\theta_0}(1-\frac{\lambda}{1+\theta_0})\log(\frac{1+\theta_0}{\lambda}-1)-\frac{\lambda}{1+\theta_0}\right]^2
&\ \mbox{for}\ \lambda\in(0,\frac{1+\theta_0}{2}]\\
\frac{1}{\lambda(1-\lambda)}\left[\frac{\lambda-\theta_0}{1-\theta_0}(1-\frac{\lambda-\theta_0}{1-\theta_0})
\log(\frac{1-\theta_0}{\lambda-\theta_0}-1)+\frac{\lambda-1}{1-\theta_0}\right]^2
&\ \mbox{for}\ \lambda\in(\frac{1+\theta_0}{2},1).
\end{array}\right.$$
Moreover
$$\left.\begin{array}{lc}
&\ \mbox{numerical values for}\ \lambda_{\mathrm{max}}\\
\theta_0=0&\ 0.22\ (\mbox{and}\ 0.78\ \mbox{by symmetry})\\
\theta_0=1/2&\ 0.92\\
\theta_0=-1/2&\ 1-0.92=0.08\ (\mbox{by symmetry})
\end{array}\right.$$

\subsection{Example \ref{ex:Pareto-distributions}}\label{sub:example4}
Here we analyze Example \ref{ex:Pareto-distributions}. For MQ
estimators we have the same rate function presented in Example
\ref{ex:scale-parameter} with $(\alpha,\omega)=(0,\infty)$ when
$G$ is as in \eqref{eq:df-Weibull} and $\rho=1$. Thus we have a
unique optimal $\lambda_{\mathrm{max}}$ which does not depend on
$\theta_0$, namely $\lambda_{\mathrm{max}}=\tilde{\lambda}_1$
where $\tilde{\lambda}_1$ is defined in Lemma
\ref{lem:particular-values}(i).

Now we briefly discuss the MM estimators for Example
\ref{ex:Pareto-distributions}. We recall that $\mu(\theta)$ is
finite only if $\theta\in\tilde{\Theta}:=(0,1)$, where
$\tilde{\Theta}\subset\Theta=(0,\infty)$. So we could consider the
mean function on the restricted parameter space $\tilde{\Theta}$,
i.e.
$$\mu(\theta)=\frac{1/\theta}{1/\theta-1}=\frac{1}{1-\theta}\ \mbox{for}\ \theta\in\tilde{\Theta}.$$
Then, if we consider the restricted parameter space
$\tilde{\Theta}$, the MM estimators
$\left\{\mu^{-1}(\bar{X}_n):n\geq 1\right\}$ are defined by
$\mu^{-1}(\bar{X}_n)=1-\bar{X}_n^{-1}$, and the function
$\mu^{-1}(\cdot)$ is continuous on $(\alpha,\omega)=(1,\infty)$.
Unfortunately we cannot apply Proposition
\ref{prop:LD-for-MM-estimators} because we cannot consider neither
the hypotheses of Proposition \ref{prop:LD-for-MM-estimators}(i)
(obvious) nor the hypotheses of Proposition
\ref{prop:LD-for-MM-estimators}(ii) because Pareto distributions
are heavy-tailed and $\Lambda_{\theta_0}^*$ is not good (see
Remark \ref{rem:goodness-hypothesis}).

\subsection{Example \ref{ex:right-endpoint-parameter}}\label{sub:example5}
Here we analyze Example \ref{ex:right-endpoint-parameter}. As far
as the MQ estimators are concerned, we can say that we cannot find
an optimal $\lambda$ because
$I_{\lambda,\theta_0}^{\prime\prime}(\theta_0)$ is an increasing
function; in fact, by
\eqref{eq:second-derivative-right-endpoint-parameter}, the
derivative of $I_{\lambda,\theta_0}^{\prime\prime}(\theta_0)$ with
respect to $\lambda$ is
$$\frac{d}{d\lambda}I_{\lambda,\theta_0}^{\prime\prime}(\theta_0)=\frac{(G^\prime(\theta_0))^2}{(1-\lambda)^2G^2(\theta_0)}.$$
We can also say that the larger is $\lambda$ the faster is the
convergence of the MQ estimators.

In the remaining part we deal with the MM estimators, and we
discuss their comparison with the MQ estimators. Obviously the
rate functions $J_{\theta_0}$ and $\Lambda_{\theta_0}^*$ are good
(we take into account Remark \ref{rem:goodness-hypothesis}) and we
can refer to the comparison between
$J_{\theta_0}^{\prime\prime}(\theta_0)$ and
$I_{\lambda,\theta_0}^{\prime\prime}(\theta_0)$ (for
$\lambda\in(0,1)$); for completeness we remark that we cannot
obtain an explicit expression of $\Lambda_{\theta_0}^*$ (even for
the simplest case with the uniform distributions, i.e. the case
$G(x)=x$ for all $x\in(0,\infty)$). It is easy to check that, if
we consider $\lambda_0$ defined by
\begin{equation}\label{eq:threshold-value-lambda-righ-endpoint}
\lambda_0:=\frac{(\mu^\prime(\theta_0))^2G^2(\theta_0)}
{(\mu^\prime(\theta_0))^2G^2(\theta_0)+\sigma^2(\theta_0)(G^\prime(\theta_0))^2},
\end{equation}
we have
$$J_{\theta_0}^{\prime\prime}(\theta_0)=\frac{(\mu^\prime(\theta_0))^2}{\sigma^2(\theta_0)}
>\frac{\lambda(G^\prime(\theta_0))^2}{(1-\lambda)G^2(\theta_0)}=I_{\lambda,\theta_0}^{\prime\prime}(\theta_0)\ \mbox{for}\ \lambda<\lambda_0$$
and
$$J_{\theta_0}^{\prime\prime}(\theta_0)=\frac{(\mu^\prime(\theta_0))^2}{\sigma^2(\theta_0)}
<\frac{\lambda(G^\prime(\theta_0))^2}{(1-\lambda)G^2(\theta_0)}=I_{\lambda,\theta_0}^{\prime\prime}(\theta_0)\
\mbox{for}\ \lambda>\lambda_0$$ by Proposition
\ref{prop:second-derivative-J} and
\eqref{eq:second-derivative-right-endpoint-parameter}.

Thus the MQ estimators converge faster than MM estimators if
$\lambda$ is close to 1; this is not surprising because the case
$\lambda=1$ concerns the case of ML estimators $\{X_{n:n}:n\geq
1\}$. For completeness we can say that, if $\{X_n:n\geq 1\}$ are
i.i.d. with distribution function $F_{\theta_0}$ as in Example
\ref{ex:right-endpoint-parameter}, the LDP in Proposition
\ref{prop:Theorem3.2-in-HMP-restricted} with $\lambda=1$ is
governed by a good rate function; thus we can consider a version
of Proposition \ref{prop:LD-for-MQ-estimators} with $\lambda=1$,
and we have the LDP of $\{X_{n:n}:n\geq 1\}$ with good rate
function $I_{1,\theta_0}$ defined by
$$I_{1,\theta_0}(\theta):=\left\{\begin{array}{ll}
\log\frac{1}{h_{\lambda,\theta_0}(\theta)}=\log\frac{G(\theta_0)}{G(\theta)}
&\ \mbox{for}\ \theta\in\Theta\ \mbox{such that}\ \theta\in(0,\theta_0)\\
\infty&\ \mbox{otherwise}.
\end{array}\right.$$

We also remark that in general the threshold value $\lambda_0$ in
\eqref{eq:threshold-value-lambda-righ-endpoint} depends on
$\theta_0$. In fact, for $G(x):=e^x-1$, after some computations we
have $\mu(\theta)=\frac{\theta e^\theta}{e^\theta-1}-1$,
$\sigma^2(\theta)=\frac{(e^\theta-1)^2-\theta^2e^\theta}{(e^\theta-1)^2}$,
and therefore
$$\lambda_0=\frac{e^{2\theta_0}-2e^{\theta_0}(1+\theta_0)+(1+\theta_0)^2}
{2e^{2\theta_0}-e^{\theta_0}(4+2\theta_0+\theta_0^2)+2+2\theta_0+\theta_0^2}.$$
Interestingly we can say that $\lambda_0$ does not depend on
$\theta_0$ if $G(x):=x^y$ for some $y>0$; in fact we have
$\mu(\theta)=\frac{y\theta}{y+1}$,
$\sigma^2(\theta)=\frac{y\theta^2}{(y+2)(y+1)^2}$, and therefore
$$\lambda_0=\frac{(\frac{y}{y+1})^2(\theta_0^y)^2}{(\frac{y}{y+1})^2(\theta_0^y)^2+\frac{y\theta_0^2}{(y+2)(y+1)^2}(y\theta_0^{y-1})^2}
=\frac{y+2}{2y+2}.$$ For instance, for the specific case of
uniform distributions cited in Example
\ref{ex:right-endpoint-parameter} (for which we have
$\mu(\theta)=\frac{\theta}{2}$ and
$\sigma^2(\theta)=\frac{\theta^2}{12}$ for all
$\theta\in(0,\infty)$; so the sequence
$\left\{\mu^{-1}(\bar{X}_n):n\geq 1\right\}$ in Proposition
\ref{prop:LD-for-MM-estimators} is defined by
$\mu^{-1}(\bar{X}_n)=2\bar{X}_n$) we have $G(x):=x$, and therefore
we get $\lambda_0=3/4$ by setting $y=1$.

Finally we remark that, in general, we cannot find $\delta>0$ such
that $I_{\lambda_0,\theta_0}(\theta)>J_{\theta_0}(\theta)$ or
$I_{\lambda_0,\theta_0}(\theta)<J_{\theta_0}(\theta)$ for
$0<|\theta-\theta_0|<\delta$; for instance (see Figure \ref{fig2}
where $\theta_0=1$) this happens for the statistical model with
uniform distributions cited above (where $G(x):=x$ and
$\lambda_0=3/4$).
\begin{figure}[ht]
\begin{center}
\includegraphics[angle=0,width=0.5\textwidth]{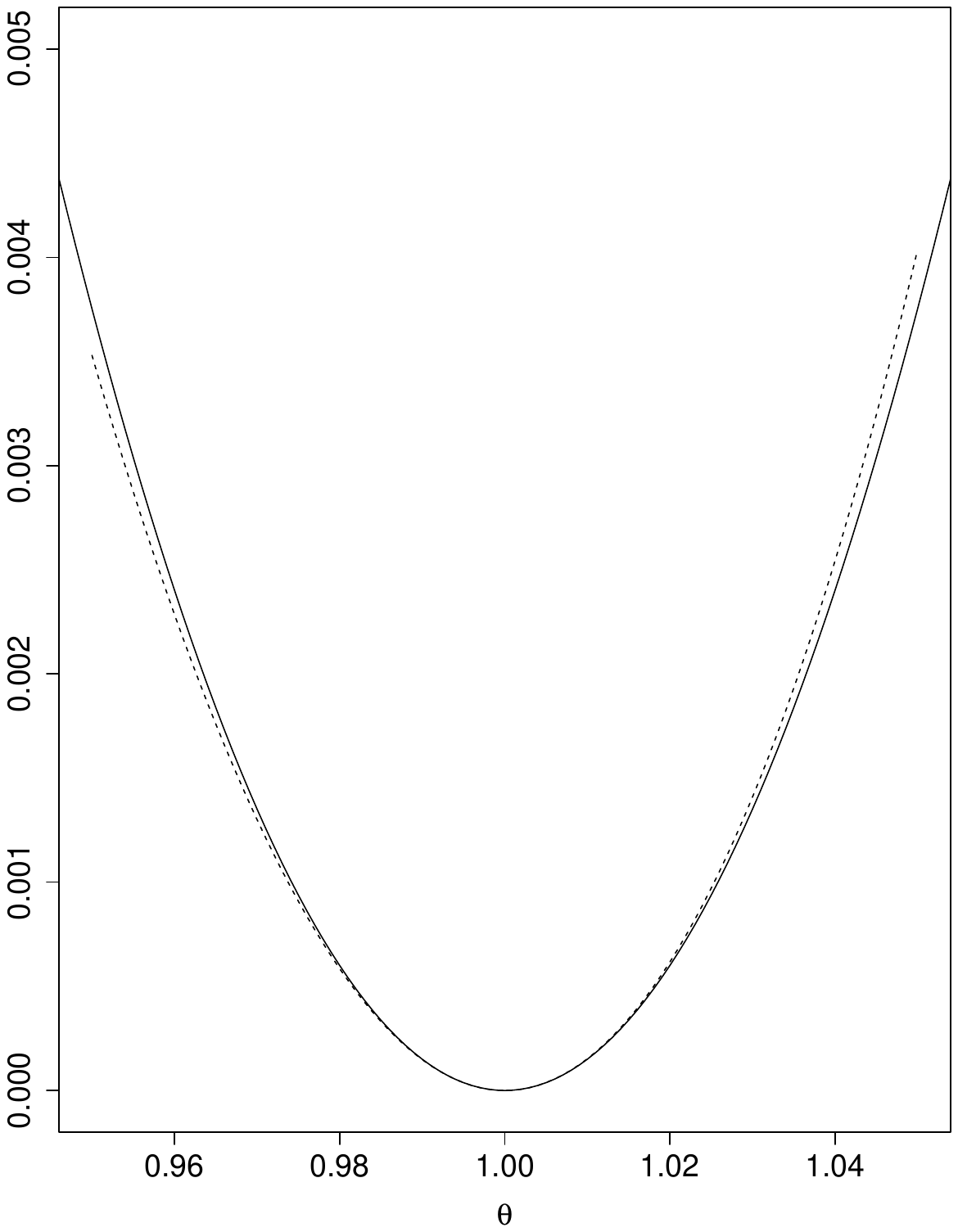}
\caption{The rate functions $I_{3/4,\theta_0}(\theta)$
(dashed line) and $J_{\theta_0}(\theta)$ (solid line) in a 
neighborhood of $\theta_0=1$ for the statistical model with 
uniform distributions.}\label{fig2}
\end{center}
\end{figure}
%
%
%
%
\appendix
\section{A class of counterexamples}\label{appendix}
In Section \ref{sub:example2}, for all the examples where the
distribution $G$ is symmetric, we find that
$G^{\prime\prime}(0)=0$ and that there is a unique optimal value
$\lambda_{\mathrm{max}}$, namely $\lambda_{\mathrm{max}}=0.5$.

Here we show that this is not necessarily the case, indeed we
present a procedure to construct another function $\tilde{G}$ with
the symmetry property and such that
$\tilde{G}^{\prime\prime}(0)=0$; this function will be determined
starting from a function $G$ with the properties cited above (for
instance it could be one of the choices illustrated in Example
\ref{ex:location-parameter} except the one concerning Gumbel
distribution). The aim is to illustrate that, for such a function
$\tilde{G}$, $\lambda=0.5$ cannot be an optimal value.

The function $\tilde{G}$ is defined by
$$\tilde{G}(x):=\left\{\begin{array}{ll}
\frac{G(1+x)}{2G^\prime(0)+1}&\ \mbox{for}\ x\leq -1\\
\frac{1}{2}+\frac{G^\prime(0)x}{2G^\prime(0)+1}&\ \mbox{for}\ |x|<1\\
\frac{2G^\prime(0)+G(x-1)}{2G^\prime(0)+1}&\ \mbox{for}\ x\geq 1.
\end{array}\right.$$
One can check that, if $G$ is twice differentiable, then
$\tilde{G}$ is also twice differentiable (and in particular the
condition $G^{\prime\prime}(0)=0$ is needed to say that
$\tilde{G}$ is twice differentiable); the details are omitted.
Moreover, if we consider
$$G(-1)=\frac{1}{2(2G^\prime(0)+1)}\ \mbox{and}\ G(1)=\frac{4G^\prime(0)+1}{2(2G^\prime(0)+1)}$$
(we recall that $G(0)=\frac{1}{2}$ by the symmetry property of
$G$), we have
$$\tilde{G}^{-1}(\lambda):=\left\{\begin{array}{ll}
G^{-1}(\lambda(2G^\prime(0)+1))-1&\ \mbox{for}\ \lambda\leq G(-1)\\
(\lambda-\frac{1}{2})\frac{2G^\prime(0)+1}{G^\prime(0)}&\ \mbox{for}\ G(-1)<\lambda<G(1)\\
G^{-1}(\lambda(2G^\prime(0)+1)-2G^\prime(0))+1&\ \mbox{for}\
\lambda\geq G(1).
\end{array}\right.$$
Then, around $\lambda=1/2$ (more precisely for
$\lambda\in(G(-1),G(1))$ because $G(-1)\in(0,\frac{1}{2})$ and
$G(1)\in(\frac{1}{2},1)$), we have
$I_{\lambda,\theta_0}^{\prime\prime}(\theta_0)=\frac{1}{\lambda(1-\lambda)}\left(\frac{G^\prime(0)}{2G^\prime(0)+1}\right)^2$
by \eqref{eq:second-derivative-location-parameter}; thus
$\lambda=0.5$ cannot be an optimal value because
$I_{\lambda,\theta_0}^{\prime\prime}(\theta_0)$ is locally
minimized at $\lambda=1/2$ (in fact $\lambda=1/2$ maximizes the
denominator $\lambda(1-\lambda)$).

\paragraph{Acknowledgements.} We thank two anonymous referees for
their useful comments.


\begin{thebibliography}{spc}
\bibitem{AitchinsonBrown}
Aitchison, J., and J.A.C. Brown. 1957. \emph{The Lognormal
Distribution}. Cambridge: Cambridge University Press.
\bibitem{CastilloHadi}
Castillo E., and A.S. Hadi. 1995. A method for estimating
parameters and quantiles of distributions of continuous random
variables. \emph{Comput. Statist. Data Anal.} 20:421--439.
\bibitem{Dasgupta}
Dasgupta, A. 2008. \emph{Asymptotic Theory of Statistics and
Probability}. New York: Springer.
\bibitem{DemboZeitouni}
Dembo A., and O. Zeitouni. 1998. \emph{Large Deviations Techniques
and Applications}. 2nd ed. New York: Springer.
\bibitem{DenuitDhaeneGoovaertsKaas}
Denuit, M., J. Dhaene, M. Goovaerts and R. Kaas. 2005.
\emph{Actuarial Theory for Dependent Risks}. Chichester: John
Wiley and Sons.
\bibitem{DominicyVeredas}
Dominicy Y., and D. Veredas. 2013. The method of simulated
quantiles. \emph{J. Econometrics} 172:235--247.
\bibitem{HashorvaMacciPacchiarotti}
Hashorva E., C. Macci, and B. Pacchiarotti. 2013. Large deviations
for proportions of observations which fall in random sets
determined by order statistics. \emph{Methodol. Comput. Appl.
Probab.} 15:875--896.
\bibitem{Hassanein69a}
Hassanein, K.M. 1969a. Estimation of the parameters of the extreme
value distribution by use of two or three order statistics.
\emph{Biometrika} 56:429--436.
\bibitem{Hassanein69b}
Hassanein, K.M. 1969b. Estimation of the parameters of the
logistic distribution by sample quantiles. \emph{Biometrika}
56:684--687.
\bibitem{Hassanein71}
Hassanein, K.M. 1971. Percentile estimators for the parameters of
the Weibull distribution. \emph{Biometrika} 58:673--676.
\bibitem{Hassanein72}
Hassanein, K.M. 1972. Simultaneous estimation of the parameters of
the extreme value distribution of sample quantiles.
\emph{Technometrics} 14:63--70.
\bibitem{Koenker2005}
Koenker, R. 2005. \emph{Quantile regression}. Cambridge: Cambridge
University Press.
\bibitem{McNeil}
McNeil, A.J., R. Frey, and P. Embrechts. 2015. \emph{Quantitative
risk management: Concepts, techniques and tools}. Princeton:
Princeton University Press.
\bibitem{MudholkarHutson}
Mudholkar G.S., and A.D. Hutson. 2000. The epsilon-skew-normal
distribution for analyzing near-normal data. \emph{J. Statist.
Plann. Inference} 83:291--309.
\bibitem{SgouropoulosYaoYastremiz}
Sgouropoulos N., Q. Yao, and C. Yastremiz. 2015. Matching a
distribution by matching quantiles estimation. \emph{J. Amer.
Statist. Assoc.} 110(510):742--759.
\end{thebibliography}
\end{document}